\documentclass{article}

\usepackage{geometry}                
\usepackage{graphicx}
\usepackage{wrapfig}
\usepackage{amssymb}
\usepackage{epstopdf}
\usepackage{enumerate}
\usepackage{gensymb}
\usepackage{amsmath}
\usepackage{amsthm}
\usepackage{ulem}
\usepackage{color}
\usepackage{amssymb}
\newtheorem{theorem}{Theorem}
\newtheorem{lemma}{Lemma}

\usepackage[utf8]{inputenc}
\usepackage[english]{babel}

\usepackage{nomencl}

\title{Space-time multiscale model reduction for transport equations}

\author{
Eric T. Chung \thanks{Department of Mathematics,
The Chinese University of Hong Kong (CUHK), Hong Kong SAR. Email: {\tt tschung@math.cuhk.edu.hk}.
The research of Eric Chung is supported by Hong Kong RGC General Research Fund (Project 14317516)
and CUHK Direct Grant for Research 2016-17.}
\and
Yalchin Efendiev \thanks{Department of Mathematics \& Institute for Scientific Computation (ISC),
Texas A\&M University,
College Station, Texas, USA. Email: {\tt efendiev@math.tamu.edu}.}
\and
Yanbo Li \thanks{Department of Mathematics, Texas A\&M University, College Station, TX 77843. Email: {\tt lyb@tamu.edu}.}
}

\begin{document}
\maketitle{}
\begin{abstract}
In this paper, we propose a space-time GMsFEM for transport equations.
Multiscale transport equations occur in many geoscientific applications,
which include subsurface transport, atmospheric pollution transport, and 
so on.
Most of existing multiscale approaches use spatial multiscale 
basis functions or upscaling, 
and there are very few works that design space-time multiscale functions
to solve the transport equation on a coarse grid.
For the time dependent problems, the use of space-time 
multiscale basis functions offers several advantages as
the spatial and temporal scales are intrinsically coupled.
By using the GMsFEM idea with a space-time framework, 
one obtains a better dimension reduction taking 
into account features of the solutions
in both space and time. 
In addition, the time-stepping can be performed using much 
coarser time step sizes compared to the case when 
spatial multiscale basis are used.
Our scheme is based on space-time snapshot spaces and
model reduction using space-time spectral problems derived from 
the analysis.
We give the analysis for the well-posedness and 
the spectral convergence of our method.
We also present some numerical examples to demonstrate the performance
of the method. In all examples, we observe a good accuracy with a few 
basis functions.
\end{abstract}

\section{\textbf{Introduction}}

Transport processes in practical applications have multiscale nature. 
The convection term in the transport equation is governed by a 
flow velocity field, which can be described by the Darcy equation or 
the steady-state Stokes equation, and the convection velocity is 
typically highly heterogeneous and contains many scales. 
Because of the spatial and magnitude variations of the velocity field,
the transport equation contains both spatial and temporal scales.
For example, high velocity fields in thin channels introduce both spatial
scales related to channel sizes and temporal scales related to 
velocity variations. These scales are tightly coupled in this example, 
as we deal with high convection in small channel like spatial regions.

Transport equations governed or dominated by convection processes
occur in many geoscientific applications. Besides subsurface processes,
the convection-dominated multiscale transport occurs in atmospheric
sciences, where particles are transported by the air. Because 
atmospheric flows can have multiple scales, one deals with multiscale
transport equations with space and time heterogeneities. Other geoscientific
applications include the transport in vadoze zone hydrology and so on.

Numerical solutions for these 
transport equations can be expensive as we need to
resolve both spatial and temporal scales. 
Some type of model reduction can be used to reduce the computational
cost and achieve a certain degree accuracy at a very reduced cost.
Model reduction techniques usually depend on a coarse grid approximation, 
which can be obtained by discretizing the problem on a coarse grid and 
choosing a suitable coarse-grid formulation of the problem. 
In the literature, several approaches have been developed to 
obtain the coarse-grid formulation, including upscaling methods 
\cite{weh02, iliev2013numerical, Iliev_fast_fibrous_10,Ewing_ILRW_SISC_09, Iliev_RW_09, Arbogast_Boyd_06, Arbogast_two_scale_04, Vass_Upscal_11, brandt05, cdgw03} and multiscale methods \cite{Efendiev_GKiL_12, Iliev_MMS_11, Efendiev_Galvis_10, EG09, ArPeWY07, Chu_Hou_MathComp_10, Arbogast_PWY_07, aej07}. 
Among these approaches, GMsFEM (Generalized Multiscale Finite Element Methods) \cite{egh12, ceh2016adaptive, AdaptiveGMsFEM, MixedGMsFEM,chung2017conservative,chen2016least} provides a systematic way of adding degrees of freedom for problems with high contrast and multiple scales.  Most of these approaches have been developed
for diffusion dominated processes. Our goal is to extend these concepts to
transport equations by designing appropriate space-time basis functions.
We note that the proposed problems do not have scale separation and
one can not represent the velocity field by a single ``average'' velocity
field on a coarse grid \cite{hou1992homogenization,weinan1992homogenization}. Appropriate number
of coarse-grid parameters is needed to obtain accurate solutions.

In this paper, we propose a space-time GMsFEM for the transport equations. 
To do so, we start with a coarse space-time grid, which does not 
necessarily resolve the fine-scale heterogeneities. 
Then, we derive a space-time discontinuous Galerkin formulation,
 which uses upwinding for the convection term
and the time derivative. The key component of the scheme is 
the basis functions, which are supported
on space-time coarse elements. 
To construct the basis functions, we apply the general concept of GMsFEM.
In particular, for each coarse space-time element, we first find 
the snapshot space.
We consider two ways to compute the snapshot space. In our first approach, 
we solve the transport equation on each space-time coarse element
with all possible initial and boundary conditions resolved 
on the underlying fine grid. 
In the second approach, we consider an oversampling strategy, in which
we solve the transport equation on oversampled space-time regions. 
Next, we perform local model reduction procedure in order 
to obtain the offline space for the computation of the solution.
In this step, we construct a local spectral problem defined 
on the snapshot space
and identify dominant modes as the basis functions. 
We remark that the spectral problem takes care both the 
space and time structures,
and is designed by our convergence analysis. 

In the paper, we will present the detailed construction of the basis 
functions. In addition, 
we will give analysis for the well-posedness of the discrete system
as well as the spectral convergence of the scheme. 
We have shown that the error is inversely 
proportional to the eigenvalues of the spectral problem.
Furthermore, we illustrate the performance of our 
scheme by a couple of test cases. 
In both cases, we see that our scheme is able 
to produce accurate solutions with only a few multiscale basis functions.
We also compare the performance with the use of space-time 
polynomial basis functions,
and show that our scheme is able to capture the scales of 
the solutions with very few degrees of freedoms. 
We remark that the use of space-time basis functions offers 
some advantages over the use of spatial multiscale basis functions.
In particular, space-time basis functions are able to capture 
the scales in both space and time, when they are tightly coupled.
The latter is the case in the applications of interest.
Besides, space-time approaches 
allow the scheme to update the solution 
with a coarser time step size. 
The success of the space-time basis functions is also 
illustrated by a work for parabolic equations \cite{chung2016spacetime}.

Numerical results are presented in the paper. 
We consider the velocity field obtained by solving flow
equation in highly heterogeneous, high-contrast media.
The media contains high-permeability channels and inclusions,
which introduce several time scales. We solve the transport equation
with some choices of boundary and initial conditions
and compare the fine-grid solution against the multiscale
solution with various number of basis functions. Our numerical
results show that with a few basis functions, we can obtain accurate
results.

The paper is organized as follows.
In Section \ref{sec:method}, we present the construction of the method,
including the space-time formulation and basis function constructions, and prove the well-posedness of the discrete system.
In Section \ref{sec:analysis}, we prove the spectral convergence of the scheme.
We illustrate the performance of the scheme by some numerical examples in Section \ref{sec:numerical}.
The paper ends with a conclusion. 

\section{\textbf{Space-time GMsFEM}}
\label{sec:method}

In this section, 
we will give the construction of our space-time GMsFEM 
for transport equations. 
First, we present some basic notations and the coarse grid formulation
in Section \ref{sec:notation}.
Then, we present the constructions of the space-time multiscale snapshot functions and
basis functions in Section \ref{sec:basis}.

\subsection{Preliminaries}\label{sec:notation}
 
Let $\Omega$ be a bounded domain in $\mathbb{R}^2$ with a Lipschitz boundary $\partial\Omega$ with unit normal vector $\mathbf{n}$,
and $[0,T]$ $(T>0)$ be a time interval. In
this paper, we consider the following transport equation:
\begin{equation}\label{main}
\begin{split}
\frac{\partial u}{\partial t}+\mathbf{v}\cdot \nabla u & =  0\quad \quad\quad \text{in}\;\Omega\times(0,T),
\\
u & =  g \quad \quad\quad\text{on}\;\Gamma^-\times(0,T),\\ 
u(x,0)& =  u_0(x)\,\quad\text{in}\;\Omega\times \left\lbrace t=0 \right\rbrace ,
\end{split}
\end{equation}
where $\mathbf{v}$ is a given divergence-free velocity field, $g$ is the inflow boundary data, $u_0(x)$ is initial condition, $\Gamma^-=\left\lbrace x\in \partial\Omega\,|\, \mathbf{v}\cdot \mathbf{n}<0 \right\rbrace $ is the inflow boundary  
and $\Gamma^+=\left\lbrace x\in \partial\Omega\,|\, \mathbf{v}\cdot \mathbf{n}>0 \right\rbrace $ is the outflow boundary.
We remark that the method presented in this paper can be applied to 3D problems.

The goal of this paper is to develop a space-time generalized multiscale finite element method. The method is motivated by the space-time finite element framework.
First, a space-time variational formulation is defined. Then some space-time multiscale basis functions are constructed.
The constructions of the multiscale basis functions follow two general steps.
In the first step, we will construct 
space-time snapshot functions in order to build a set of possible modes
of the solution. The snapshot functions are obtained by solving 
local problem on coarse space-time cells. 
We consider
 the use of oversampling technique with the aim of reducing the offline cost. 
In the second step, we will construct multiscale basis functions. 
To do so, we design a suitable spectral problem defined in the snapshot space, 
and use the first few dominant eigenfunctions as the basis functions. 
We note that the spectral problems take both space and time into account. 
By using the space-time multiscale basis functions, we obtain a reduced model
which takes into account the variations of the solutions in both space and time,
and thus produces accurate 
solution for the transport equation in heterogeneous media. 


{\color{black}{
Let $\mathcal{T}^h$ be a partition
of the domain $\Omega$ into fine finite elements. 
Here $h>0$ is the fine mesh size. 
The coarse partition,
$\mathcal{T}^H$
of the domain $\Omega$, is formed such that
each element in $\mathcal{T}^H$ is a connected 
union of fine-grid blocks.
More precisely,
 $\forall K_j \in \mathcal{T}^H,\;K_j=\cup_{F\in I_j}F$ 
for some $I_j\subset \mathcal{T}^h$. 
The quantity $H>0$ is the coarse mesh size. 
We will consider rectangular coarse elements and the methodology 
can be used with general coarse elements. An illustration of the mesh notations is shown in Figure \ref{partition} (left).

\begin{figure}[ht]
	\begin{minipage}[t]{0.5\textwidth}
		\centering
		\includegraphics[width=2.6in]{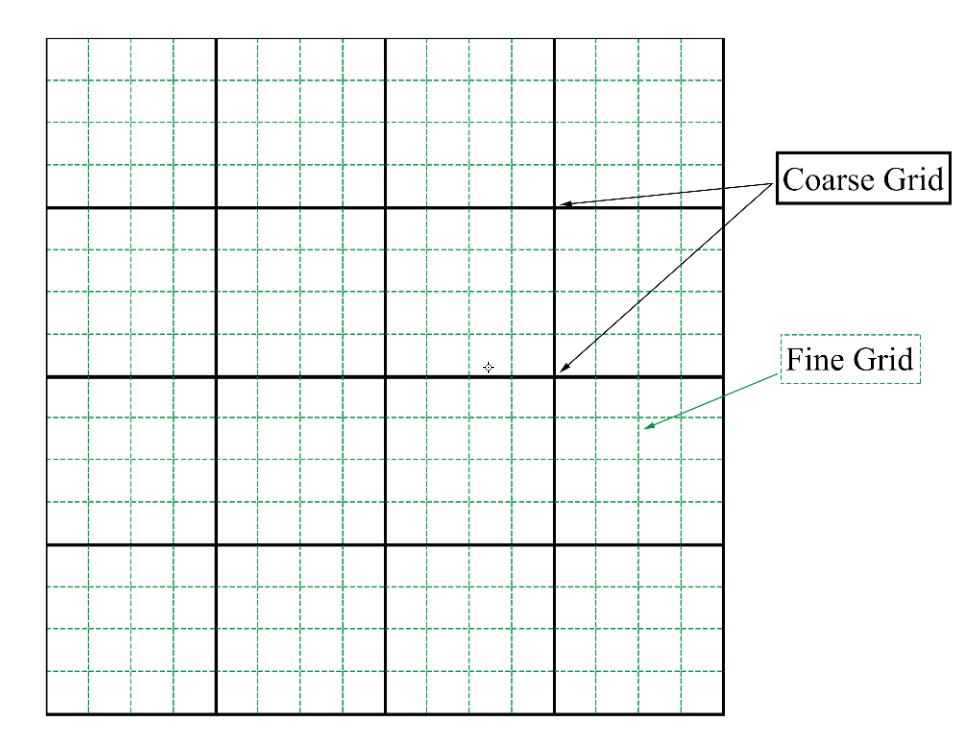}
	\end{minipage}%
	\begin{minipage}[t]{0.5\textwidth}
		\centering
		\includegraphics[width=3.3in]{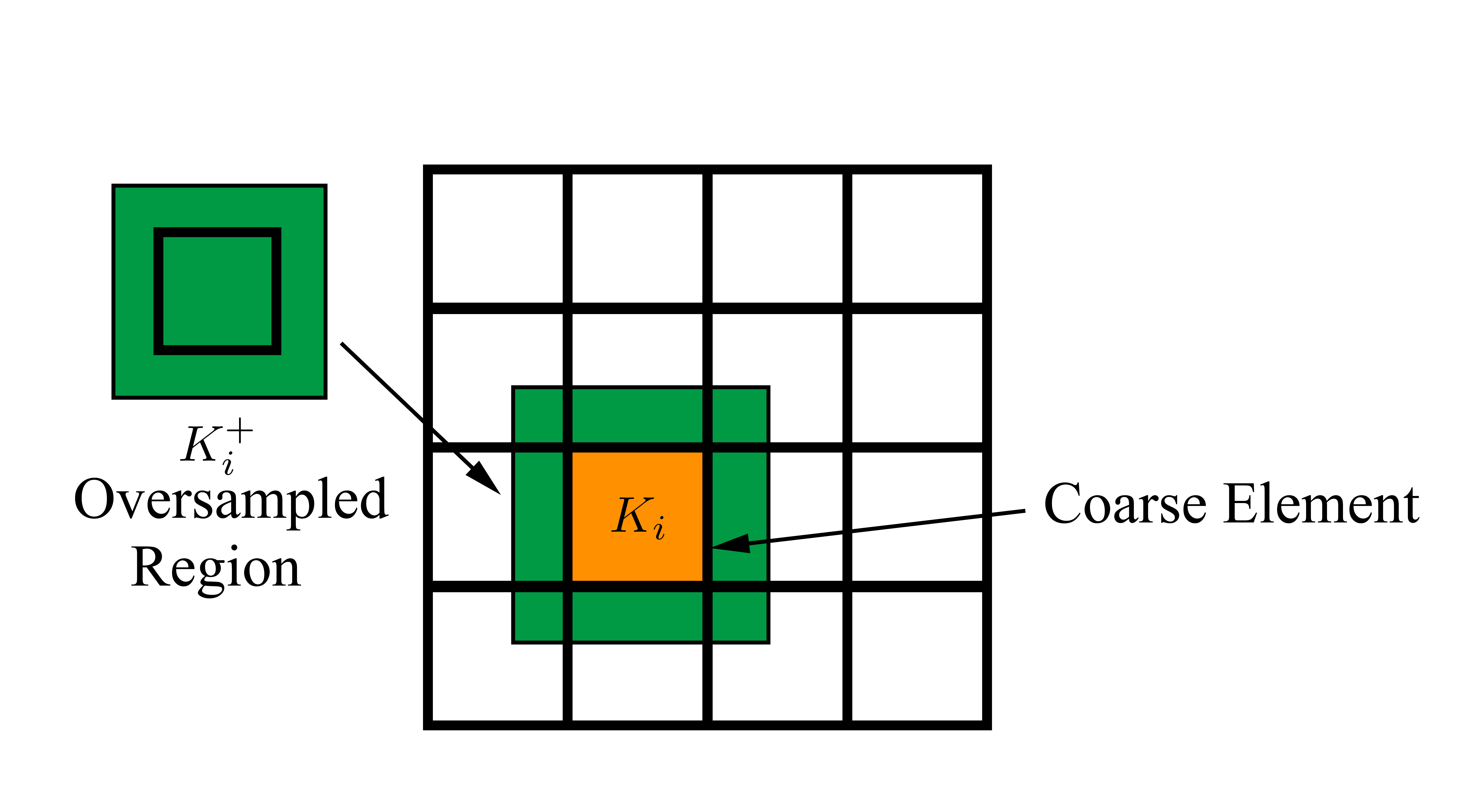}
	\end{minipage}
	\caption{Left: an illustration of fine and coarse grids. Right: an illustration of a coarse neighborhood and a coarse element.}
	\label{partition}
\end{figure}

Next, let $\mathcal{T}^T=\{(T_{n-1},T_n)|1\leq n\leq N\}$ be a coarse partition of $(0,T)$, where
$$0=T_0<T_1<T_2<\cdots<T_N=T$$
and we define a fine partition $\mathcal{T}^t$ of $(0,T)$ by refining the partition $\mathcal{T}^T$, that is  $\mathcal{T}^t=\{(T_{i-\frac{1}{r}},T_{i})|i=\frac{1}{r},\frac{2}{r},\cdots,N-\frac{1}{r},N\}$, where
$$0=T_0<T_{\frac{1}{r}}<\cdots<T_{1-\frac{1}{r}}<T_1<T_{1+\frac{1}{r}}<\cdots<T_N=T$$

To fix the notations, we define the finite element space $V_h$ with respect to $\mathcal{T}^h\times(0,T)$ as a space consists of piecewise linear functions in fine grid. Here we introduce two types of $V_h$.

\subsubsection{CG in coarse cell}\label{sec:CG}
We use the term "coarse cell" to represent $K\times (T_{n-1},T_n)$ where $K$ is a coarse element in space, and $(T_{n-1},T_n)$ is a coarse time interval. In this case, all functions in $V_h$ are continuous in each coarse cell, that is
\begin{align*}
V_h= & \left\{ v\in L^2((0,T);\Omega)\,|\, v=\phi(x)\psi(t)\text{ where } \phi|_K\in Q_1(K)\;\forall K\in \mathcal{T}^h,\,\phi|_K\in C^0(K)\;\forall K\in \mathcal{T} ^H, \right.\\
& \left. \psi|_\tau\in P_1(\tau) \;\forall \tau\in \mathcal{T}^t,\,\text{ and } \psi|_\tau\in C^0(\tau)\;\forall \tau\in \mathcal{T} ^T   \right\}.
\end{align*}

Next, we let $\mathcal{E}_H$ be the collection of all coarse edges, and $\mathcal{E}_H^0=\mathcal{E}_H\backslash \partial \Omega$. For the value on a coarse edge, which is shared by two coarse blocks $K_i$ and $K_j$, if $K_i$ is the upwind block, define $w^+=w|_{K_i}$ and $w^-=w|_{K_j}$ for the corresponding downwind value. Figure \ref{upwind} gives an illustration. 

\begin{figure}[ht]
	\begin{minipage}[t]{1\textwidth}
		\centering
		\includegraphics[width=2in]{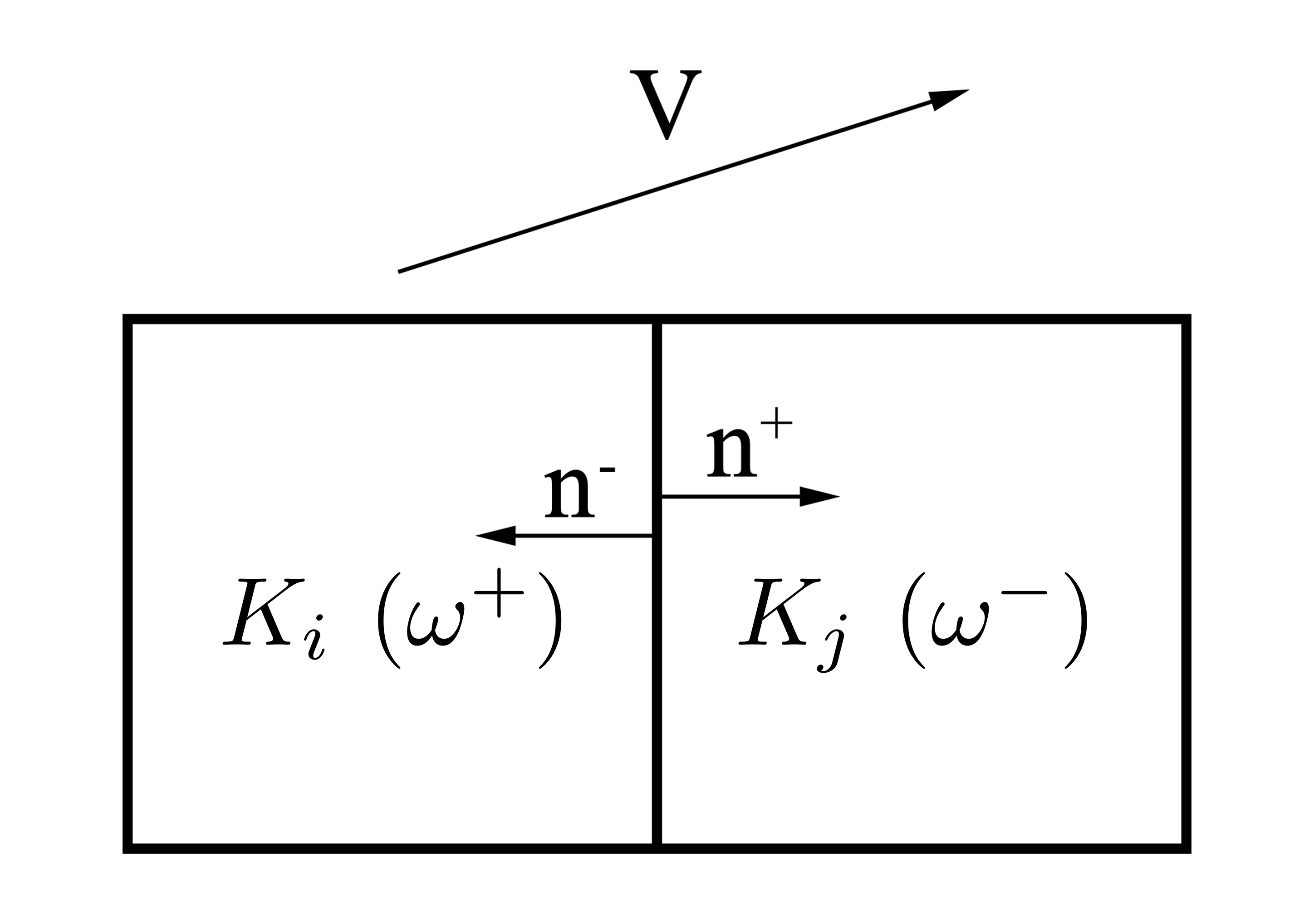}
	\end{minipage}%
	\caption{An illustration of upwind and downwind blocks.}
	\label{upwind}
\end{figure}

The fine-scale solution $u_h\in V_h$ is obtained by solving the following variational problem
\begin{equation}
\label{eq:fine_CG}
\begin{split}
\int_0^T\int_\Omega \left ( \frac{\partial u_h}{\partial t}w-u_h\nabla w\cdot\mathbf{v}  \right )+ \int_0^T\sum_{e\in\mathcal{E}_H^0}\int_e u_h^+[w]\cdot\mathbf{v}+\int_0^T\sum_{e\in{\Gamma^+}}\int_e u_h w\mathbf{v}\cdot\mathbf{n}\\
-\sum_{n=0}^{N-1}\int_\Omega \left [\left [u_h(x,T_n)  \right ]  \right ]w(x,T_n^+)=\int_0^T u_0(x)w(x,T_n^+)-\int_0^T\sum_{e\in{\Gamma^-}}\int_e gw\mathbf{v}\cdot\mathbf{n}, \quad\forall w\in V_h.
\end{split}
\end{equation}
where $[\cdot]$ is the jump operator in space defined by
\begin{equation}\label{space_jump}
[w]=\left\{\begin{matrix}
w^-\mathbf{n}^-+w^+\mathbf{n}^+   & \text{on }\mathcal{E}_H^0,\\
w^-\mathbf{n}^-   & \text{on }\Gamma^-,\\
w^+\mathbf{n}^+   & \text{on }\Gamma^+.
\end{matrix}\right.
\end{equation}
And $[[\cdot]]$ is the jump operator in time defined by
\begin{equation}\label{time_jump}
\left [\left [u_h(x,T_n)  \right ]  \right ]=\left\{\begin{matrix}
u_h(x,T_0^+)   & \text{for }n=0,\\
u_h(x,T_n^+)-u_h(x,T_n^-)   & \text{for }n>0.
\end{matrix}\right.
\end{equation}

The above equation uses an upwind approximation in $\mathbf{v} \cdot \nabla u$ term, and is motivated by \cite{chung2016spacetime} and \cite{riviere2008discontinuous}.
We assume that the fine mesh size $h$ is small enough so that the fine-scale solution $u_h$ is close enough to the
exact solution. 
We will skip the discussion on the well-posedness of (\ref{eq:fine_CG}) since it is similar to that of the coarse scale system to be presented. 
We will also skip the approximation property of the fine-scale solver since it is standard (see for example \cite{riviere2008discontinuous}).

We note that the purpose of this paper is to find a multiscale solution $u_H$ that is a good approximation
of the fine-scale solution $u_h$.

Now we present the general idea of GMsFEM. We will use the space-time finite element method to solve
the system \eqref{main} on the coarse grid.
We will use a similar framework as (\ref{eq:fine_CG}).
That is, we find $u_H\in V_H$ such that
\begin{align}\label{global 0_CG}
\begin{split}
\int_0^T\int_\Omega \left ( \frac{\partial u_H}{\partial t}w-u_H\nabla w\cdot\mathbf{v}  \right )+ \int_0^T\sum_{e\in\mathcal{E}_H^0}\int_e u_H^+[w]\cdot\mathbf{v}+\int_0^T\sum_{e\in{\Gamma^+}}\int_e u_H w\mathbf{v}\cdot\mathbf{n}\\
+\sum_{n=0}^{N-1}\int_\Omega \left [\left [u_H(x,T_n)  \right ]  \right ]w(x,T_n^+)=\int_{\Omega}u_0(x)w(x,T_0^+)-\int_0^T\sum_{e\in{\Gamma^-}}\int_e gw\mathbf{v}\cdot\mathbf{n}\quad\forall w\in V_H,
\end{split}
\end{align}
where $V_H$ is the multiscale finite element space which will be introduced in the following subsections.\\
{\color{black}{
		\quad
		To avoid a large computational cost associated
		with solving the equation \eqref{global 0_CG}, we divide
		the computational domain.
		We assume the solution space $V_H$ is a direct sum of the
		spaces only containing the functions defined on one single coarse time interval $(T_{n-1},T_n)$. We decompose
		the problem \eqref{global 0_CG} into a sequence of problems and find the solution $u_H$ in each time interval sequentially. Our
		coarse space will be constructed in each time interval
		$$V_H=\bigoplus_{n=1}^N V_H^{(n)},$$
		where $V_H^{(n)}$ only contains the functions having
		zero values in the time interval $(0,T)$ except $(T_{n-1},T_n)$,
		namely $\forall v \in V_H^{(n)}$,
		$$v(\cdot,t)=0 \text{ for }t\in(0,T)\backslash(T_{n-1},T_n).$$
}}

The equation \eqref{global 0_CG} can be decomposed into the following problem: find $u_H^{(n)}\in V_H^{(n)}$ (where $V_H^{(n)}$ will be
defined later) satisfying

\begin{equation}\label{global_CG}
\begin{split}
&\int_{T_{n-1}}^{T_n}\int_\Omega \left ( \frac{\partial u_H^{(n)}}{\partial t}w-u_H^{(n)}\nabla w\cdot\mathbf{v}  \right )+\int_{T_{n-1}}^{T_n}\sum_{e\in\mathcal{E}_H^0}\int_e u_H^{(n)+}[w]\cdot\mathbf{v}\\
&+
\int_{T_{n-1}}^{T_n}\sum_{e\in{\Gamma^+}}\int_e u_H^{(n)}w\mathbf{v}\cdot\mathbf{n}
+\int_\Omega u_H^{(n)}(x,T_{n-1}^+)w(x,T_{n-1}^+)\\
=&\int_\Omega f_H^{(n)}(x)w(x,T_{n-1}^+)
-\int_{T_{n-1}}^{T_n}\sum_{e\in{\Gamma^-}}\int_e gw\mathbf{v}\cdot\mathbf{n},\quad\forall w\in V_H^{(n)},
\end{split}
\end{equation}
where
$$f_H^{(n)}(x)=\left\{\begin{matrix}
u_H^{(n-1)}(x,T_{n-1}^-) & \text{for} & n\geq 2 \\
u_0 & \text{for} & n=1.
\end{matrix}\right.$$

Then the solution $u_H$ is the direct sum of all these $u_H^{(n)}$'s, that is $u_H=\bigoplus_{n=1}^N u_H^{(n)}.$
Next, we motivate the use of space-time multiscale basis functions by
comparing it to space multiscale
basis functions.
Let $\left\{T_{n-1},T_{n-1+\frac{1}{r}},\cdots,T_{n-\frac{1}{r}},T_n\right\}$ be $r$ fine time steps
in $(T_{n-1},T_n)$. The solution can be represented as
$u_H^{(n)}=\sum_{l,i}c_{l,i}\phi_l^{K_i}(x,t)$ in the interval $(T_{n-1},T_n)$,
where the number of coefficients $c_{l,i}$ is related to
the size of the reduced system in space-time interval.
If we use spatial multiscale
basis functions,
these multiscale basis functions are constructed
at each fine time interval $(T_{p-\frac{1}{r}},T_p)$, denoted
by $\phi_l^{K_i}(x,T_p)$. The solution $u_H$ spanned
by these basis functions will have a larger dimension
since
each time interval is represented by multiscale basis functions.

\subsubsection{DG in coarse cell}\label{sec:DG}
In this case, all functions in $V_h$ could be discontinuous in
each coarse cell, that is
\begin{align*}
V_h=  \left\{ v\in L^2((0,T);\Omega)\,|\, v=\phi(x)\psi(t)\text{ where } \phi|_K\in Q_1(K)\;\forall K\in \mathcal{T}^h,\, \psi|_\tau\in P_1(\tau) \;\forall \tau\in \mathcal{T}^t   \right\}.
\end{align*}
}}

We let $\mathcal{E}_h$ be the collection of all fine edges, and $\mathcal{E}_h^0=\mathcal{E}_h\backslash \partial \Omega$.
The fine-scale solution $u_h\in V_h$ is obtained by solving the following variational problem
\begin{equation}
\label{eq:fine_DG}
\begin{split}
\int_0^T\int_\Omega \left ( \frac{\partial u_h}{\partial t}w-u_h\nabla w\cdot\mathbf{v}  \right )+ \int_0^T\sum_{e\in\mathcal{E}_h^0}\int_e u_h^+[w]\cdot\mathbf{v}+\int_0^T\sum_{e\in{\Gamma^+}}\int_e u_h w\mathbf{v}\cdot\mathbf{n}\\
+\sum_{p=0}^{N-\frac{1}{r}}\int_\Omega \left [\left [u_h(x,T_p)  \right ]  \right ]w(x,T_p^+)=\int_{\Omega}u_0(x)w(x,T_0^+)-\int_0^T\sum_{e\in{\Gamma^-}}\int_e gw\mathbf{v}\cdot\mathbf{n}, \quad\forall w\in V_h,
\end{split}
\end{equation}
where the jump operators $[\cdot]$ and $\left[[\cdot] \right]$ have similar definition to equation \eqref{space_jump} and \eqref{time_jump}.

As for GMsFEM, we find $u_H\in V_H$ such that
\begin{align}\label{global 0_DG}
\begin{split}
\int_0^T\int_\Omega \left ( \frac{\partial u_H}{\partial t}w-u_H\nabla w\cdot\mathbf{v}  \right )+ \int_0^T\sum_{e\in\mathcal{E}_h^0}\int_e u_H^+[w]\cdot\mathbf{v}+\int_0^T\sum_{e\in{\Gamma^+}}\int_e u_H w\mathbf{v}\cdot\mathbf{n}\\
+\sum_{p=0}^{N-\frac{1}{r}}\int_\Omega \left [\left [u_H(x,T_p)  \right ]  \right ]w(x,T_p^+)=\int_{\Omega}u_0(x)w(x,T_0^+)-\int_0^T\sum_{e\in{\Gamma^-}}\int_e gw\mathbf{v}\cdot\mathbf{n}, \quad\forall w\in V_h,
\end{split}
\end{align}

The equation \eqref{global 0_DG} can be decomposed into the following problem: find $u_H^{(n)}\in V_H^{(n)}$ (where $V_H^{(n)}$ will be
defined later) satisfying

\begin{equation}\label{global_DG}
\begin{split}
&\int_{T_{n-1}}^{T_n}\int_\Omega \left ( \frac{\partial u_H^{(n)}}{\partial t}w-u_H^{(n)}\nabla w\cdot\mathbf{v}  \right )+\int_{T_{n-1}}^{T_n}\sum_{e\in\mathcal{E}_h^0}\int_e u_H^{(n)+}[w]\cdot\mathbf{v}+
\int_{T_{n-1}}^{T_n}\sum_{e\in{\Gamma^+}}\int_e u_H^{(n)}w\mathbf{v}\cdot\mathbf{n}
\\
&+\int_\Omega u_H^{(n)}(x,T_{n-1}^+)w(x,T_{n-1}^+)+\sum_{p=n-1+\frac{1}{r}}^{n-\frac{1}{r}}\int_\Omega \left [\left [u_H^{(n)}(x,T_p)  \right ]  \right ]w(x,T_p^+)\\
=&\int_\Omega f_H^{(n)}(x)w(x,T_{n-1}^+)
-\int_{T_{n-1}}^{T_n}\sum_{e\in{\Gamma^-}}\int_e gw\mathbf{v}\cdot\mathbf{n},\quad\forall w\in V_H^{(n)},
\end{split}
\end{equation}

\subsection{Construction of offline basis functions}\label{sec:basis}

In this section, we will give the constructions of multiscale basis functions. 
In Section \ref{sec:snap}, we will present the construction of the snapshot space. 
To do so, we will solve the transport equation on coarse space-time cells with suitable
initial and boundary conditions. This process will provide a set of functions which
are able to span the fine-scale solution with high accuracy. 
We will also consider the use of the oversampling technique by solving the transport equation
on a domain larger then the target coarse space-time cell. 
Next, in Section \ref{sec:off}, we will present the construction of our multiscale basis functions. 
The construction is based on the design of a local spectral problem which can identify important modes
in the snapshot space. 
Our choice of spectral problem is based on our convergence analysis given later.

\subsubsection{\textbf{Snapshot Space}} \label{sec:snap}
Let $K_i$ be a given coarse element in space. Consider the coarse time interval $(T_{n-1},T_{n})$.
We will construct a snapshot space $V_\text{snap}^{i(n)}$ containing functions defined on coarse cell $K_i \times (T_{n-1},T_{n})$. 
A spectral problem is then solved in the snapshot space to extract the dominant modes in the snapshot space. These dominant modes are the offline multiscale basis functions and the resulting reduced space is called the offline space. 
We will present two choices of $V_\text{snap}^{i(n)}$.

The first
choice for the snapshot spaces consists of solving the transport equation on the target space-time coarse cell $K_i \times (T_{n-1},T_{n})$
for all possible boundary conditions. In
particular, we define the $j$-th snapshot function $\psi_j$ as the solution to the following problem
\begin{align}\label{snapshot1}
\left\{\begin{array}{rl}
\frac{\partial \psi_j}{\partial t}+\mathbf{v}\cdot\nabla\psi_j =0 & \text{in}  \quad K_i\times(T_{n-1},T_{n}),\\ 
\psi_j(x,t) = \delta_{ij}(x,t) & \text{on} \quad \partial(K_i\times(T_{n-1},T_{n})).
\end{array}\right.
\end{align}
Here $\delta_{ij}(x,t)$ is a fine-grid delta function and $\partial(K_i\times(T_{n-1},T_{n}))$ denotes the boundaries $t=T_{n-1}$ and on $(\partial K_i)^-\times(T_{n-1},T_{n})$. Then $V_\text{snap}^{i(n)}$ consists of all $\psi_j$'s.

To improve the accuracy of the solution, we can take an advantage of oversampling concepts. We denote by $K_i^+$ the oversampled space region of $K_i \subset K_i^+$, defined by adding several fine- or coarse-grid layers around $K_i$ 
(see Figure \ref{partition}). Also, we define $(T_{n-1}^*,T_{n})$ as the left-side oversampled time region for $(T_{n-1},T_{n})$. We generate our second choice of the 
snapshot space on the oversampled space-time region $K_i^+\times(T_{n-1}^*,T_{n})$ by solving
\begin{align}\label{snapshot2}
\left\{\begin{array}{rl}
\frac{\partial \psi_j^+}{\partial t}+\mathbf{v}\cdot\nabla\psi_j^+= 0 & \text{in}  \quad K_i^+\times(T_{n-1}^*,T_{n})\\ 
\psi_j^+(x,t) = \delta_{ij}(x,t) & \text{on}  \quad \partial(K_i^+\times(T_{n-1}^*,T_{n})).
\end{array}\right.
\end{align}
	Then $V_\text{snap}^{i(n)+}$ consists of all $\psi_j^+$'s, and $V_\text{snap}^{i(n)}$ consists of all $\psi_j=\psi_j^+|_{K_i}$'s.
	Finally, $V_\text{snap}^{(n)}$ is spanned by all functions in each $V_\text{snap}^{i(n)}$, that is
	$$V_\text{snap}^{(n)}=\bigoplus_{K_i}V_\text{snap}^{i(n)}.$$
	We will use the second choice of the snapshot space in the rest of the paper.
	
	For the case in Section \ref{sec:CG}, we define snapshot solution $u_\text{snap}^{(n)}\in V_\text{snap}^{(n)}$ such that
	
	\begin{equation}\label{snap_CG}
	\begin{split}
	&\int_{T_{n-1}}^{T_n}\int_\Omega \left ( \frac{\partial u_\text{snap}^{(n)}}{\partial t}w-u_\text{snap}^{(n)}\nabla w\cdot\mathbf{v}  \right )+\int_{T_{n-1}}^{T_n}\sum_{e\in\mathcal{E}_H^0}\int_e u_\text{snap}^{(n)+}[w]\cdot\mathbf{v}+
	\int_{T_{n-1}}^{T_n}\sum_{e\in{\Gamma^+}}\int_e u_\text{snap}^{(n)}w\mathbf{v}\cdot\mathbf{n}
	\\
	&+\int_\Omega u_\text{snap}^{(n)}(x,T_{n-1}^+)w(x,T_{n-1}^+)+\sum_{n=0}^{N-1}\int_\Omega \left [\left [u_\text{snap}^{(n)}(x,T_n)  \right ]  \right ]w(x,T_n^+)\\
	=&\int_\Omega f_\text{snap}^{(n)}(x)w(x,T_{n-1}^+)
	-\int_{T_{n-1}}^{T_n}\sum_{e\in{\Gamma^-}}\int_e gw\mathbf{v}\cdot\mathbf{n},\quad\forall w\in V_\text{snap}^{(n)},
	\end{split}
	\end{equation}
	where
	$$f_\text{snap}^{(n)}(x)=\left\{\begin{matrix}
	u_\text{snap}^{(n-1)}(x,T_{n-1}^-) & \text{for} & n\geq 2 \\
	u_0 & \text{for} & n=1.
	\end{matrix}\right.$$
	And for the case in Section \ref{sec:DG}, we define snapshot solution $u_\text{snap}^{(n)}\in V_\text{snap}^{(n)}$ such that
	
	\begin{equation}\label{snap_DG}
	\begin{split}
	&\int_{T_{n-1}}^{T_n}\int_\Omega \left ( \frac{\partial u_\text{snap}^{(n)}}{\partial t}w-u_\text{snap}^{(n)}\nabla w\cdot\mathbf{v}  \right )+\int_{T_{n-1}}^{T_n}\sum_{e\in\mathcal{E}^0}\int_e u_\text{snap}^{(n)+}[w]\cdot\mathbf{v}+
	\int_{T_{n-1}}^{T_n}\sum_{e\in{\Gamma^+}}\int_e u_\text{snap}^{(n)}w\mathbf{v}\cdot\mathbf{n}
	\\
	&+\int_\Omega u_\text{snap}^{(n)}(x,T_{n-1}^+)w(x,T_{n-1}^+)+\sum_{p=n-1+\frac{1}{r}}^{n-\frac{1}{r}}\int_\Omega \left [\left [u_\text{snap}^{(n)}(x,T_p)  \right ]  \right ]w(x,T_p^+)\\
	=&\int_\Omega f_\text{snap}^{(n)}(x)w(x,T_{n-1}^+)
	-\int_{T_{n-1}}^{T_n}\sum_{e\in{\Gamma^-}}\int_e gw\mathbf{v}\cdot\mathbf{n},\quad\forall w\in V_\text{snap}^{(n)},
	\end{split}
	\end{equation}
\subsubsection{\textbf{Offline Space}}\label{sec:off}
To obtain the offline multiscale basis functions, we need to perform a space reduction by appropriate spectral problems.
Motivated by our later convergence analysis, we adopt the following spectral problem on $K_i^+\times(T_{n-1}^*,T_{n})$: 
 find $\left( \phi^+,\lambda\right) \in V_\text{snap}^{i(n)+}\times \mathbb{R}$ such that
$$a_n(\phi^+,\eta^+)=\lambda s_n(\phi^+,\eta^+), \;\;\;\forall \eta^+\in V_\text{snap}^{i(n)+},$$
where the bilinear operators $a_n(\phi^+,\eta^+)$ and $s_n(\phi^+,\eta^+)$ are defined as follow:\\
For the case in Section \ref{sec:CG}:
\begin{align*}
a_n(\phi^+,\eta^+)=&\int_{T_{n-1}^*}^{T_n}\int_{K_i^+}  \nabla \phi^+\cdot\nabla \eta^+,\\
s_n(\phi^+,\eta^+)=&\frac{1}{2}\left ( \int_{K_i}\phi^+(x,T_{n-1}^+)\eta^+(x,T_{n-1}^+)+\int_{K_i}\phi^+(x,T_{n}^-)\eta^+(x,T_{n}^-)+\int_{T_{n-1}}^{T_{n}}\int_{\partial {K_i}}\left | \mathbf{v}\cdot\mathbf{n} \right |\phi^+\eta^+ \right).
\end{align*}
For the case in Section \ref{sec:DG}:
\begin{align*}
a_n(\phi,\eta)=&\frac{1}{2}\left(\sum_{p=n-1+\frac{1}{r}}^{n-\frac{1}{r}}\int_{K_i}\left[ \left[ \phi(x,T_p)\right] \right] \left[ \left[ \eta(x,T_p)\right] \right] +\int_{T_{n-1}}^{T_n}\sum_{e\in \mathcal{E}^0(K_i)}\int_e\left | \mathbf{v}\cdot\mathbf{n} \right |[\phi][\eta]\right) ,\\
s_n(\phi,\eta)=&\frac{1}{2}\left ( \int_{K_i}\phi(x,T_{n-1}^+)\eta(x,T_{n-1}^+)+\int_{K_i}\phi(x,T_{n}^-)\eta(x,T_{n}^-)\right. \\
&\left. +\sum_{p=n-1+\frac{1}{r}}^{n-\frac{1}{r}}\int_{K_i}\left(\phi(x,T_p^-)\eta(x,T_p^-)+\phi(x,T_p^+)\eta(x,T_p^+) \right)  +\int_{T_{n-1}}^{T_{n}}\sum_{\tau\subset K_i}\int_{\partial \tau}\left | \mathbf{v}\cdot\mathbf{n} \right |\phi\eta \right).
\end{align*}

We arrange the eigenfunctions $\phi_j^+$'s in ascending order of the corresponding eigenvalues $\lambda_j$'s, and obtain $\phi_j$'s on the target region $K_i\times(T_{n-1},T_n)$ by restricting $\phi_j^+$'s onto $K_i\times(T_{n-1},T_n)$.
Then we select first $L_i$ functions $\phi_1,\phi_2,...,\phi_{L_i}$ to construct local offline space $V_H^{i(n)}$, and perform POD to remove linearly dependent functions. We define $L=\max_i L_i$. 
Finally $V_H^{(n)}$ is spanned by all functions in each $V_H^{i(n)}$, that is
$$V_H^{(n)}=\bigoplus_{K_i}V_H^{i(n)}.$$
This is the approximation space we used to solve the system (\ref{main})
using the formulation (\ref{global_DG}). 

\section{\textbf{Convergence Analysis}}
\label{sec:analysis}

In this section, we will analyze the convergence of our proposed method. We will only consider the case in Section \ref{sec:CG}, the case in Section \ref{sec:DG} will be similar.

First, we will define the following norms
\begin{equation*}
\left \| u \right \|_{V^{(n)}}^2=\frac{1}{2}\left ( \int_\Omega u^2(x,T_{n-1}^+)+\int_\Omega u^2(x,T_{n}^-)+\int_{T_{n-1}}^{T_n}\sum_{e\in \mathcal{E}_H}\int_e\left | \mathbf{v}\cdot\mathbf{n} \right |[u]^2\right )
\end{equation*}
and
\begin{equation*}
\left\| u\right\| _{W^{(n)}}^2=\frac{1}{2}\left ( \int_\Omega u^2(x,T_{n}^-)+\int_\Omega u^2(x,T_{n-1}^+)+\int_{T_{n-1}}^{T_{n}} \sum_{K_i}\int_{\partial K_i}\left| \mathbf{v}\cdot\mathbf{n}\right| u^2 \right ).
\end{equation*}
We will first show that the problem (\ref{global_CG}) is well-posed. 
Then we will prove a best approximation property.
Finally, we will prove an error bound of our method.
To begin our convergence analysis, we write (\ref{global_CG}) as
\begin{equation*}
a(u_H^{(n)},w) = F(w)
\end{equation*}
where 
\begin{equation*}
\begin{split}
a(u,w)=&\int_{T_{n-1}}^{T_n}\int_\Omega \left ( \frac{\partial u}{\partial t}w-u\nabla w\cdot\mathbf{v}  \right )+ 
\int_{T_{n-1}}^{T_n}\sum_{e\in\mathcal{E}_H^0}\int_e u^+[w]\cdot\mathbf{v}\\
& +\int_{T_{n-1}}^{T_n}\sum_{e\in{\Gamma^+}}\int_e uw\mathbf{v}\cdot\mathbf{n}+\int_\Omega u(x,T_{n-1}^+)w(x,T_{n-1}^+)
\end{split}
\end{equation*}
and
\begin{equation*}
F(w) = \int_\Omega f_H^{(n)}(x)w(x,T_{n-1}^+)-\int_{T_{n-1}}^{T_n}\sum_{e\in{\Gamma^-}}\int_e gw\mathbf{v}\cdot\mathbf{n}.
\end{equation*}

In the following theorem, we prove the well-posedness of the scheme (\ref{global_CG}).

\begin{theorem}\label{th1}
	The space-time GMsFEM (\ref{global_CG}) has a unique solution. In addition, we have the following coercivity result 
	\begin{equation*}
	a(u,u)=\left \| u \right \|_{V^{(n)}}^2, \quad \forall u \in V_{\text{snap}}^{(n)}.
	\end{equation*}
\end{theorem}
\begin{proof}
	Since the system (\ref{global_CG}) is a square linear system, 
	it suffices to prove that if $a(\widehat{u},w)=0$ for any $w\in V_H^{(n)}$, then $\widehat{u}=0$.
	To prove this, we will show that $a(u,u)=\left \| u \right \|_{V^{(n)}}^2$ for all $u \in V_{\text{snap}}^{(n)}$.

	By direct calculations, we have 
	\begin{align*}
	a(u,u)=&\int_{T_{n-1}}^{T_n}\int_\Omega \left ( \frac{\partial u}{\partial t}u-u\nabla u\cdot\mathbf{v}  \right )+ 
	\int_{T_{n-1}}^{T_n}\sum_{e\in\mathcal{E}_H^0}\int_e u^+[u]\cdot\mathbf{v}\\
	& +\int_{T_{n-1}}^{T_n}\sum_{e\in{\Gamma^+}}\int_e u^2\mathbf{v}\cdot\mathbf{n}+\int_\Omega u^2(x,T_{n-1}^+)\\
	=&\frac{1}{2}\int_\Omega \left( u^2(x,T_{n}^-)-u^2(x,T_{n-1}^+)\right)- \frac{1}{2}\int_{T_{n-1}}^{T_n}\sum_{K_i}\int_{\partial K_i}u^2\mathbf{v}\cdot\mathbf{n}\\
	&+\int_{T_{n-1}}^{T_n}\sum_{e\in\mathcal{E}_H^0}\int_e u^+[u]\cdot\mathbf{v}+\int_{T_{n-1}}^{T_n}\sum_{e\in{\Gamma^+}}\int_e u^2\mathbf{v}\cdot\mathbf{n}+\int_\Omega u^2(x,T_{n-1}^+)\\
	=&\frac{1}{2}\int_\Omega \left( u^2(x,T_{n}^-)+u^2(x,T_{n-1}^+)\right)- \frac{1}{2}\int_{T_{n-1}}^{T_n}\sum_{K_i}\int_{\partial K_i}u^2\mathbf{v}\cdot\mathbf{n}\\
	&+\int_{T_{n-1}}^{T_n}\sum_{e\in\mathcal{E}_H^0}\int_e u^+[u]\cdot\mathbf{v}+\int_{T_{n-1}}^{T_n}\sum_{e\in{\Gamma^+}}\int_e u^2\mathbf{v}\cdot\mathbf{n}.
	\end{align*}
	Since
	\begin{align*}
	&-\frac{1}{2}\int_{T_{n-1}}^{T_n}\sum_{K_i}\int_{\partial K_i}u^2\mathbf{v}\cdot\mathbf{n}+\int_{T_{n-1}}^{T_n}\sum_{e\in\mathcal{E}_H^0}\int_e u^+[u]\cdot\mathbf{v}+\int_{T_{n-1}}^{T_n}\sum_{e\in{\Gamma^+}}\int_e u^2\mathbf{v}\cdot\mathbf{n}\\
	=&-\frac{1}{2}\int_{T_{n-1}}^{T_n}\sum_{e\in{\Gamma^+}}\int_e\left| \mathbf{v}\cdot\mathbf{n}\right|u^2+\frac{1}{2}\int_{T_{n-1}}^{T_n}\sum_{e\in{\Gamma^-}}\int_e\left| \mathbf{v}\cdot\mathbf{n}\right|u^2+\frac{1}{2}\int_{T_{n-1}}^{T_n}\sum_{e\in\mathcal{E}_H^0}\int_e\left| \mathbf{v}\cdot\mathbf{n}\right|\left({u^-} ^2-{u^+} ^2\right) \\
	&+\int_{T_{n-1}}^{T_n}\sum_{e\in\mathcal{E}_H^0}\int_e \left| \mathbf{v}\cdot\mathbf{n}\right|u^+\left( u^+-u^-\right) +\int_{T_{n-1}}^{T_n}\sum_{e\in{\Gamma^+}}\int_e \left| \mathbf{v}\cdot\mathbf{n}\right|u^2\\
	=&\frac{1}{2}\int_{T_{n-1}}^{T_n}\sum_{e\in{\Gamma^+}}\int_e\left| \mathbf{v}\cdot\mathbf{n}\right|u^2+\frac{1}{2}\int_{T_{n-1}}^{T_n}\sum_{e\in{\Gamma^-}}\int_e\left| \mathbf{v}\cdot\mathbf{n}\right|u^2+\frac{1}{2}\int_{T_{n-1}}^{T_n}\sum_{e\in\mathcal{E}_H^0}\int_e\left| \mathbf{v}\cdot\mathbf{n}\right|\left({u^+}-{u^-}\right)^2\\
	=&\frac{1}{2}\int_{T_{n-1}}^{T_n}\sum_{e\in \mathcal{E}_H}\int_e\left | \mathbf{v}\cdot\mathbf{n} \right |[u]^2,
	\end{align*}
	we obtain $a(u,u)=\left \| u \right \|_{V^{(n)}}^2$. In particular, $a(\widehat{u},\widehat{u})=\left \| \widehat{u} \right \|_{V^{(n)}}^2$.
	By assumption that $a(\widehat{u},w)=0$ for any $w\in V_H^{(n)}$, 
	we have $\left \| \widehat{u} \right \|_{V^{(n)}}^2=0$.
	So, $\widehat{u} (x,T_{n-1}^+)=\widehat{u}(x,T_{n}^-)=0$, $\left | \mathbf{v}\cdot\mathbf{n} \right|\widehat{u}=0$ on $e\in \partial \Omega$, and $\left | \mathbf{v}\cdot\mathbf{n} \right|\widehat{u}^-=\left | \mathbf{v}\cdot\mathbf{n} \right|\widehat{u}^+$ on $e\in \mathcal{E}_H^0$. Then,
	for any $t_0\in(T_{n-1},T_n)$, from equation \eqref{snapshot2}, we have
	$$\int_{T_{n-1}}^{t_0}\int_\Omega(\frac{\partial \widehat{u}}{\partial t}+\mathbf{v}\cdot \nabla \widehat{u})\widehat{u}=\int_{T_{n-1}}^{t_0}\sum_{K_i}\int_{K_i}(\frac{\partial \widehat{u}}{\partial t}+\mathbf{v}\cdot \nabla \widehat{u})\widehat{u}=0.$$
	On the other hand, using integration by parts, we have
	\begin{align*}
	&\int_{T_{n-1}}^{t_0}\int_\Omega(\frac{\partial \widehat{u}}{\partial t}+\mathbf{v}\cdot \nabla \widehat{u})\widehat{u}\\
	=&\frac{1}{2}\int_\Omega \widehat{u}^2(x,t_0)-\frac{1}{2}\int_\Omega \widehat{u}^2(x,T_{n-1}^+)+\frac{1}{2}\int_{T_{n-1}}^{t_0}\sum_{K_i}\int_{\partial K_i}\widehat{u}^2\mathbf{v}\cdot\mathbf{n}\\
	=&\frac{1}{2}\int_\Omega \widehat{u}^2(x,t_0).
	\end{align*}
	Thus $\widehat{u}(x,t_0)=0$ for any $t_0\in(T_{n-1},T_n)$, that is $\widehat{u}=0$. Hence, we proved the theorem.
\end{proof}

In the following, we will prove a best approximation result. In particular, 
we will show that the $V^{(n)}$-norm of the error $u_\text{snap}-u_H$ can be bounded by the $W^{(n)}$-norm of
the difference $u_\text{snap}-w$ for any $w\in V_H^{(n)}$ plus the error from the previous time step. 
\begin{lemma}	\label{lemma1}
	Let $u_\text{snap}$ be the snapshot solution of (\ref{snap_CG}) and let $u_H$ be the multiscale solution of (\ref{global_CG}). Then
	we have the following estimate
	$$\left \| u_\text{snap}-u_H \right \|_{V^{(n)}}^2\leq
	C\inf_{w\in V_H^{(n)}}\left \| u_\text{snap}-w \right \|_{W^{(n)}}^2+\left \| u_\text{snap}-u_H \right \|_{V^{(n-1)}}^2,$$
	where $C$ is a constant independent of the velocity $\bf{v}$ and the mesh size.
\end{lemma}
\begin{proof}
	We will first show the boundedness condition 
	$a(u,w)\leq \sqrt{2} \left\| u\right\| _{V^{(n)}}\left\| w\right\| _{W^{(n)}}$ for any $u,w \in V_{\text{snap}}^{(n)}$.
	Notice that, using integration by parts and (\ref{snapshot2}), we have 	
	\begin{align*}
	&\int_{T_{n-1}}^{T_n}\int_\Omega \left ( \frac{\partial u}{\partial t}w-u\nabla w\cdot\mathbf{v}  \right )\\
	=&\int_{T_{n-1}}^{T_n}\int_\Omega \left ( \frac{\partial u}{\partial t}w+w\nabla u\cdot\mathbf{v}  \right )-\int_{T_{n-1}}^{T_{n}} \sum_{K_i}\int_{\partial K_i}\mathbf{v}\cdot\mathbf{n}uw\\
	=&\int_{T_{n-1}}^{T_n}\sum_{K_i}\int_{K_i} \left ( \frac{\partial u}{\partial t}+\nabla u\cdot\mathbf{v}  \right )w-\int_{T_{n-1}}^{T_{n}} \sum_{K_i}\int_{\partial K_i}\mathbf{v}\cdot\mathbf{n}uw\\
	=&-\int_{T_{n-1}}^{T_{n}} \sum_{K_i}\int_{\partial K_i}\mathbf{v}\cdot\mathbf{n}uw.
	\end{align*}
	Therefore, we have 
	\begin{align*}
	a(u,w)=&-\int_{T_{n-1}}^{T_{n}} \sum_{K_i}\int_{\partial K_i}\mathbf{v}\cdot\mathbf{n}uw+\int_{T_{n-1}}^{T_n}\sum_{e\in\mathcal{E}_H^0}\int_e u^+[w]\cdot\mathbf{v}+\int_{T_{n-1}}^{T_n}\sum_{e\in{\Gamma^+}}\int_e uw\mathbf{v}\cdot\mathbf{n}\\
	&+\int_\Omega u(x,T_{n-1}^+)w(x,T_{n-1}^+)\\
	=&-\int_{T_{n-1}}^{T_n}\sum_{e\in\mathcal{E}_H^0}\int_e w^-[u]\cdot\mathbf{v}-\int_{T_{n-1}}^{T_n}\sum_{e\in{\Gamma^-}}\int_e uw\mathbf{v}\cdot\mathbf{n}+\int_\Omega u(x,T_{n-1}^+)w(x,T_{n-1}^+)\\
	\leq&\left (\int_\Omega u^2(x,T_{n-1}^+)+\int_{T_{n-1}}^{T_n}\sum_{e\in\mathcal{E}_H^0}\int_e \left | \mathbf{v}\cdot\mathbf{n} \right |[u]^2+\int_{T_{n-1}}^{T_n}\sum_{e\in{\Gamma^-}}\int_e \left | \mathbf{v}\cdot\mathbf{n} \right |u^{2}  \right )^{1/2}\\
	&\left (\int_\Omega w^2(x,T_{n-1}^+)+\int_{T_{n-1}}^{T_n}\sum_{e\in\mathcal{E}_H^0}\int_e \left | \mathbf{v}\cdot\mathbf{n} \right |{w^{-}}^2+\int_{T_{n-1}}^{T_n}\sum_{e\in{\Gamma^-}}\int_e \left | \mathbf{v}\cdot\mathbf{n} \right |w^{2}  \right )^{1/2}\\
	\leq&\sqrt{2}\left \| u \right \|_{V^{(n)}}\left (\int_\Omega w^2(x,T_{n-1}^+)+\int_{T_{n-1}}^{T_n}\sum_{e\in\mathcal{E}_H^0}\int_e \left | \mathbf{v}\cdot\mathbf{n} \right |{w^{-}}^2+\int_{T_{n-1}}^{T_n}\sum_{e\in{\Gamma^-}}\int_e \left | \mathbf{v}\cdot\mathbf{n} \right |w^{2}  \right )^{1/2}.
	\end{align*}
	We will next estimate the right hand side of the above inequality. 
	From equation (\ref{snapshot2}), we have
	\begin{align*}
	0=&\int_{T_{n-1}}^{T_n}\int_\Omega\left ( \frac{\partial w}{\partial t}+\mathbf{v}\cdot \nabla w \right )w\\
	=&\frac{1}{2}\left ( \int_\Omega w^2(x,T_{n}^-)-\int_\Omega w^2(x,T_{n-1}^+)+\int_{T_{n-1}}^{T_{n}} \sum_{K_i}\int_{\partial K_i}\mathbf{v}\cdot\mathbf{n}w^2 \right )\\
	=&\frac{1}{2}\left ( 
	\int_\Omega w^2(x,T_{n}^-)-\int_\Omega w^2(x,T_{n-1}^+)-\int_{T_{n-1}}^{T_{n}} \sum_{e\in \Gamma^-}\int_{e}\left |\mathbf{v}\cdot\mathbf{n}  \right |w^2 \right.\\
	&\left.+ 
	\int_{T_{n-1}}^{T_{n}} \sum_{e\in \Gamma^+}\int_{e}\left |\mathbf{v}\cdot\mathbf{n}  \right |w^2+\int_{T_{n-1}}^{T_{n}} \sum_{e\in \mathcal{E}_H^0}\int_{e}\left |\mathbf{v}\cdot\mathbf{n}  \right |\left ({w^{+}}^2-{w^{-}}^2  \right ) \right ).
	\end{align*}
	Thus, we obtain 
	\begin{align*}
	&\int_\Omega w^2(x,T_{n-1}^+)+\int_{T_{n-1}}^{T_n}\sum_{e\in\mathcal{E}_H^0}\int_e \left | \mathbf{v}\cdot\mathbf{n} \right |{w^{-}}^2+\int_{T_{n-1}}^{T_n}\sum_{e\in{\Gamma^-}}\int_e \left | \mathbf{v}\cdot\mathbf{n} \right |w^{2}  \\
	=&\frac{1}{2}\left ( 
	\int_\Omega w^2(x,T_{n}^-)+\int_\Omega w^2(x,T_{n-1}^+)+\int_{T_{n-1}}^{T_{n}} \sum_{e\in \Gamma^-}\int_{e}\left |\mathbf{v}\cdot\mathbf{n}  \right |w^2 \right.\\
	&\left.+ 
	\int_{T_{n-1}}^{T_{n}} \sum_{e\in \Gamma^+}\int_{e}\left |\mathbf{v}\cdot\mathbf{n}  \right |w^2+\int_{T_{n-1}}^{T_{n}} \sum_{e\in \mathcal{E}_H^0}\int_{e}\left |\mathbf{v}\cdot\mathbf{n}  \right |\left (w^{+2}+w^{-2}  \right ) \right )\\
	=&\frac{1}{2}\left ( \int_\Omega w^2(x,T_{n}^-)+\int_\Omega w^2(x,T_{n-1}^+)+\int_{T_{n-1}}^{T_{n}} \sum_{K_i}\int_{\partial K_i}\left| \mathbf{v}\cdot\mathbf{n}\right| w^2 \right )\\
	=&\left\| w\right\| _{W^{(n)}}^2.
	\end{align*}
	So, we have proved the desired inequality. 
	
	Next, using the coercivity and the boundedness of the bilinear form $a(v,w)$, we obtain the following best approximation result
	\begin{equation}\label{best approximation}
	\left\| u_\text{snap}-u_H\right\| _{V^{(n)}}\leq \sqrt{2} \left\| u_\text{snap}-w\right\| _{W^{(n)}}\;\;\;\forall w\in V_H^{(n)}.
	\end{equation}
	Combining \eqref{global_CG} and a similar formulation for the fine-scale solution $u_\text{snap}$, for any $v\in V_H^{(n)}$, we have
	\begin{align}
	a(u_\text{snap}-u_H,v)=&\int_\Omega \left( f_\text{snap}^{(n)}(x)-f_H^{(n)}(x)\right) v(x,T_{n-1}^+)\\
	=&\, \int_\Omega \left( u_\text{snap}(x,T_{n-1}^-)-u_H(x,T_{n-1}^-)\right) v(x,T_{n-1}^+)\\
	\leq&\,2 \left\| u_\text{snap}-u_H\right\| _{V^{(n-1)}}\left\| v\right\| _{V^{(n)}}. \label{eq:ineq1}
	\end{align}
	Therefore for any $w\in V_H^{(n)}$, setting $v=w-u_H\in V_H^{(n)}$ , and using the coercivity, boundedness and the above best approximation result, we obtain
	\begin{align*}
	& \left \| u_\text{snap}-u_H \right \|_{V^{(n)}}^2\\
	=&\, a(u_\text{snap}-u_H,u_\text{snap}-u_H)\\
	=&\, a(u_\text{snap}-u_H,u_\text{snap}-w)+a(u_\text{snap}-u_H,w-u_H).
	\end{align*}
	Using (\ref{eq:ineq1}), we have 
	\begin{align*}
	& \left \| u_\text{snap}-u_H \right \|_{V^{(n)}}^2\\
	\leq&\, \sqrt{2} \left \| u_\text{snap}-u_H \right \|_{V^{(n)}}\left \| u_\text{snap}-w \right \|_{W^{(n)}}+2\left \| u_\text{snap}-u_H \right \|_{V^{(n-1)}}\left \| w-u_H \right \|_{V^{(n)}}\\
	\leq&\, 2\left \| u_\text{snap}-w \right \|_{W^{(n)}}^2+\left \| u_\text{snap}-u_H \right \|_{V^{(n-1)}}^2  \\
	&\, +\frac{1}{2}\left \| w-u_\text{snap}+u_\text{snap}-u_H \right \|_{V^{(n)}}^2\\
	\leq&\, 2\left \| u_\text{snap}-w \right \|_{W^{(n)}}^2+\left \| u_\text{snap}-u_H \right \|_{V^{(n-1)}}^2 \\
	&\, +\frac{1}{2}\left( \left ( 1+\sqrt{2} \right )\left \| u_\text{snap}-w \right \|_{V^{(n)}}^2+\left ( 1+\frac{1}{\sqrt{2}} \right )\left \| u_\text{snap}-u_H \right \|_{V^{(n)}}^2\right) \\
	\leq&\, 2\left \| u_\text{snap}-w \right \|_{W^{(n)}}^2+\left \| u_\text{snap}-u_H \right \|_{V^{(n-1)}}^2 \\
	&\, +\frac{1}{2}\left( \left ( 1+\sqrt{2} \right )\left \| u_\text{snap}-w \right \|_{V^{(n)}}^2+\left ( 1+\frac{1}{\sqrt{2}} \right )2\left \| u_\text{snap}-w \right \|_{W^{(n)}}^2\right) \\
	=&\, \frac{9+2\sqrt{2}}{2}\left \| u_\text{snap}-w \right \|_{W^{(n)}}^2+\left \| u_\text{snap}-u_H \right \|_{V^{(n-1)}}^2.
	\end{align*}
	Hence, we proved the lemma.
\end{proof}

Now, we are ready to prove our main convergence result in this section. First, we define some notations.
For any fine-scale function $u_\text{snap} \in V_\text{snap}$, we can write $u_\text{snap}=\sum_{i}u_{\text{snap},i}$ where $u_{\text{snap},i}\in V_\text{snap}^{i(n)}$
and the sum is taken over all spatial coarse elements $K_i$. 
We remark that this representation holds for each coarse time interval. 
Since the snapshot functions are the restriction of solutions of the transport equation on oversampled regions, we can write
$u_{\text{snap},i}=u_{\text{snap},i}^+|_{K_i\times (T_{n-1},T_{n})}$ where $u_{\text{snap},i}^+\in V_\text{snap}^{i(n)+}$.
The following is our main spectral convergence theorem.
\begin{theorem}
	Let $u_\text{snap}$ be the fine-scale solution of (\ref{snap_CG}) and let $u_H$ be the multiscale solution of (\ref{global_CG}). 
	Then we have
	$$ \left \| u_\text{snap}-u_H \right \|_{V^{(n)}}^2\leq \frac{C}{\Lambda_*}\sum_{i}a_n\left( u_{\text{snap},i}^+,u_{\text{snap},i}^+\right)+\left \| u_\text{snap}-u_H \right \|_{V^{(n-1)}}^2.$$
	where $\Lambda_*=\min_i \lambda_{L_i+1}^{(i)}$.
\end{theorem}
\begin{proof}
	Note that $u_\text{snap}=\sum_{i}u_{\text{snap},i}=\sum_{i}\sum_{l}c_{l,i}\phi_l^i$, where $\phi_l^i$ is the $l$-th multiscale basis function for the coarse element $K_i$.
	Using this expression, we can define a projection of $u_\text{snap}$ into $V_H^{(n)}$ by
	$$P(u_\text{snap})=\sum_{i}\sum_{l\leq L_i}c_{l,i}\phi_l^i.$$
	Then we have 
	\begin{align*}
	\inf_{w\in V_H^{(n)}}\left \| u_\text{snap}-w \right \|_{W^{(n)}}^2\leq&\left \| u_\text{snap}-P(u_\text{snap}) \right \|_{W^{(n)}}^2\\
	=&\sum_{i}s_n\left( u_{\text{snap},i}-P(u_{\text{snap},i}),u_{\text{snap},i}-P(u_{\text{snap},i})\right) \\
	=&\sum_{i}s_n\left( u_{\text{snap},i}^+-P(u_{\text{snap},i})^+,u_{\text{snap},i}^+-P(u_{\text{snap},i})^+\right) \\
	\leq&\sum_{i}\frac{1}{\lambda_{L_i+1}^{(i)}}a_n\left( u_{\text{snap},i}^+,u_{\text{snap},i}^+\right)\\
	\leq&\frac{1}{\Lambda_*}\sum_{i}a_n\left( u_{\text{snap},i}^+,u_{\text{snap},i}^+\right)
	.
	\end{align*}
	Combining with Lemma \ref{lemma1}, we proved the theorem.
\end{proof}

	Let $u$ be the exact solution to problem \eqref{main}. We also note that $u_\text{snap}\approx u$ when $h$ is small enough.\\
	Similar to the proof of \eqref{best approximation}, we can prove
	$$\left\| u-u_\text{snap}\right\| _{V^{(n)}}\leq \sqrt{2} \left\| u-w\right\| _{W^{(n)}}\;\;\;\forall w\in V_\text{snap}^{(n)}.$$
	In particular, we choose $w=\widetilde{u}\in V_\text{snap}^{(n)}$ such that  $\widetilde{u}=P(g)\text{ on }\Gamma^-\times(T_{n-1}-T_n)$ and $\widetilde{u}(x,T_{n-1})=P(u(x,T_{n-1})),$ where $P$ is some piecewise linear interpolation.
	Hence $u-\widetilde{u}$ is the solution to the following equation
	\begin{equation*}
	\begin{split}
	\frac{\partial (u-\widetilde{u})}{\partial t}+\mathbf{v}\cdot \nabla (u-\widetilde{u}) & =  0\quad \quad\quad \text{in}\;\Omega\times(0,T),
	\\
	(u-\widetilde{u}) & =  (I-P)(g) \quad \quad\quad\text{on}\;\Gamma^-\times(0,T),\\ 
	(u-\widetilde{u})(x,0)& =  (I-P)(u_0(x))\,\quad\text{in}\;\Omega\times \left\lbrace t=0 \right\rbrace ,
	\end{split}
	\end{equation*}
	Since $(I-P)(g)$ and $(I-P)(u_0(x))$ converge to $0$ when $h$ converges to $0$, we can regard $u\approx\widetilde{u}$ when $h$ is small enough. Hence $u_\text{snap}\approx u$ when $h$ is small enough.
\section{\textbf{Numerical Results}}
\label{sec:numerical}

In this section, we present several
numerical examples for the case in Section \ref{sec:CG} to show the performance of the proposed method. The situation in Section \ref{sec:DG} will be similar.
We solve the system \eqref{main} using the space-time 
GMsFEM.
The space domain $\Omega$ is taken as the unit square $[0,1]\times[0,1]$ and is divided into $10\times10$ coarse blocks consisting of uniform squares. Each coarse block is then divided into $10\times10$ fine blocks consisting of uniform squares. That is, $\Omega$ is partitioned by $100\times100$ square fine blocks. The whole time interval is $(0,0.08)$ (i.e., $T=0.08$) and is divided into $80$ uniform coarse time 
intervals and each coarse time interval is then divided into $5$ fine time intervals. And we define an oversampling region $K_i^+\times (T_{n-1}^*,T_{n})$ by enlarging $K_i\times (T_{n-1},T_{n})$ by one coarse grid layer. 

\subsection{Example 1} 
In our first example, we consider CG in coarse cell case, 
take $u_0=\sin(2x+2y)$ and $g=\sin(2x+2y-4t)$. To generate a heterogeneous divergence-free velocity field $\mathbf{v}=(v_1,v_2)$, we solve the following high contrast flow equation using a fine-scale mixed method:
\begin{align*}
\left\{\begin{matrix}
\kappa^{-1}\mathbf{v}+\nabla p &=0 &\quad \text{in}\; \Omega, \\ 
\nabla \cdot\bf{v}&=0 & \quad \text{in}\; \Omega,\\ 
\bf{v}\cdot{n}&=f &\quad \text{on}\; \partial\Omega,
\end{matrix}\right.
\end{align*}
where 
\begin{align*}
f=
\left\{\begin{matrix}
-1 & \text{on}\; \{0\}\times (0,1),\\
1 & \text{on}\; \{1\}\times (0,1),\\
0 & \text{otherwise},
\end{matrix}\right.
\end{align*}
and $\kappa$ is a heterogeneous media. 
The heterogeneous field $\kappa$ and and the corresponding velocity $\bf{v}$ are shown in Figure \ref{prameter}.

\begin{figure}[ht]
	\begin{minipage}[t]{0.33\textwidth}
		\centering
		\includegraphics[width=2.05in]{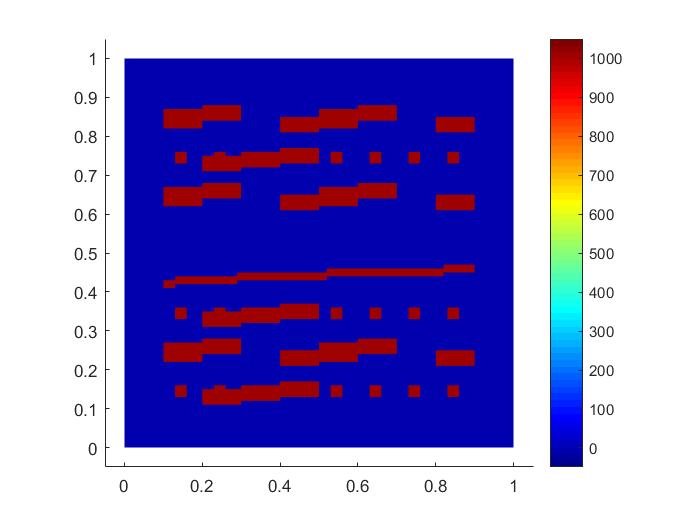}
		\title{$\kappa$}
	\end{minipage}
	\begin{minipage}[t]{0.33\textwidth}
		\centering
		\includegraphics[width=2.05in]{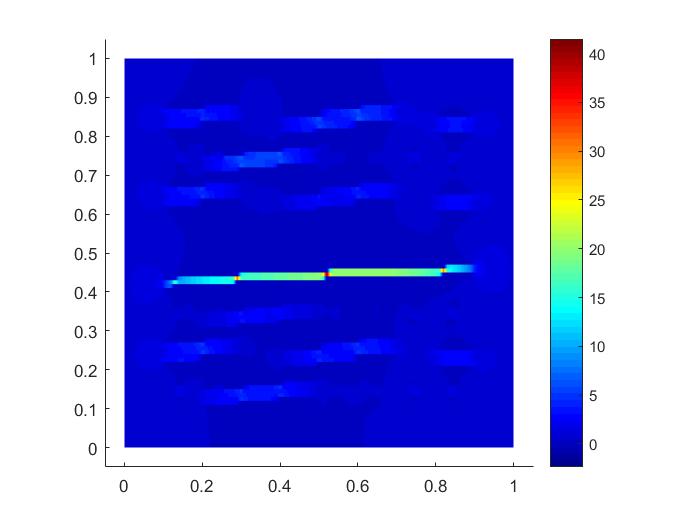}
		\title{$x$-component of velocity $\bf{v}$}
	\end{minipage}%
	\begin{minipage}[t]{0.33\textwidth}
		\centering
		\includegraphics[width=2.05in]{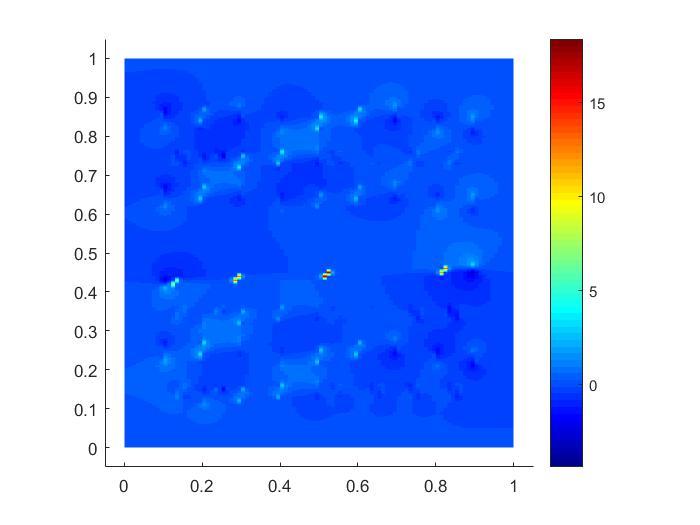}
		\title{$y$-component of velocity $\bf{v}$}
	\end{minipage}
	\caption{A heterogeneous field $\kappa$ and the corresponding velocity $\bf{v}$.}
	\label{prameter}
\end{figure}

To compare the accuracy, we will use the following error quantities:
$$e_1=\left ( \frac{\int_0^T\int_\Omega\left | u_H-u_h \right |^2}{\int_0^T\int_\Omega\left | u_h \right |^2} \right )^{1/2},\quad\quad e_2=\left ( \frac{\int_\Omega\left | u_H(\cdot,T)-u_h(\cdot,T) \right |^2}{\int_\Omega\left | u_h(\cdot,T) \right |^2} \right )^{1/2}.$$

Furthermore, we introduce the concept of \textit{snapshot ratio}:
$$\text{snapshot ratio}=\frac{\text{dim}(V_H^{(n)})}{\text{dim}(V_\text{snap}^{(n)})},$$
where $\text{dim}(V_H^{(n)})$ refers to the dimension of offline space, and $\text{dim}(V_\text{snap}^{(n)})$ refers to the number of functions $\delta_{ij}(x,t)$ from equation \eqref{snapshot2}.

\begin{figure}[ht] 
  \centering
  \includegraphics[width=2.5in]{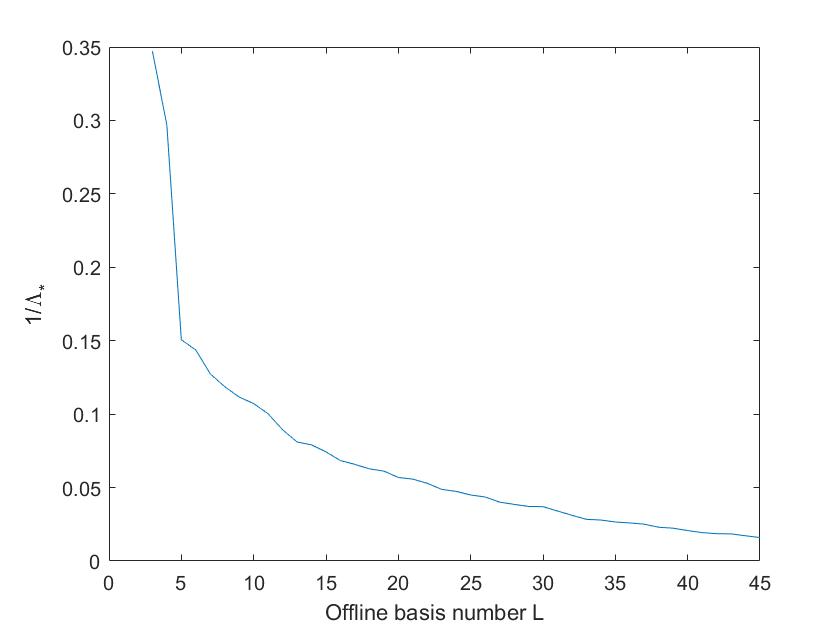} 
  \caption{The values $1/\Lambda_*$ against number of basis functions.}
  \label{inverse lambda}
\end{figure}

\begin{table}[ht]
	\centering  %
	\begin{tabular}{|c|c|c|c|c|}
		\hline
		\hline
		$L$ & dim$(V_H^{(n)})$ & snapshot ratio & $e_1$ & $e_2$  \\
		\hline
		1 &100 & 0.45\%  &  45.85\% & 48.38\% \\
		\hline
		3 &300 & 1.34\%  &  7.29\% & 10.08\%  \\
		\hline
		5 &500 & 2.24\%  &  6.01\% & 8.41\%  \\
		\hline
		7 &700 & 3.13\%  &  4.22\% & 5.73\%  \\
		\hline
		10 &1000 & 4.47\%  &  3.48\% & 4.99\%\\
		\hline
		15 &1500 & 6.71\%  &  2.83\% & 4.24\% \\
		\hline
		20 &2000 & 8.94\% &  2.46\% & 3.64\%  \\
		\hline
		25 &2500 & 11.18\% &  2.07\% & 3.16\%  \\
		\hline
		30 &3000 & 13.41\%  &  1.85\% & 2.82\% \\
		\hline
	\end{tabular}
	\caption{Errors for Example 1 (dim$(V_h^{(n)})$=72600 and dim$(V_\text{snap}^{(n)})$=22365 for each time step $n$).}
	\label{error table }
\end{table}

\begin{figure}[ht]
	\begin{minipage}[t]{0.45\textwidth}
		\centering
		\includegraphics[width=3in]{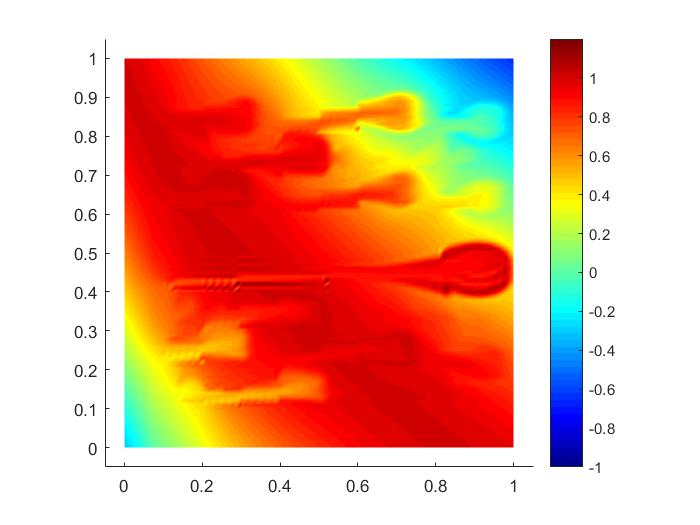}
		\title{$u_h$}
	\end{minipage}%
	\begin{minipage}[t]{0.45\textwidth}
		\centering
		\includegraphics[width=3in]{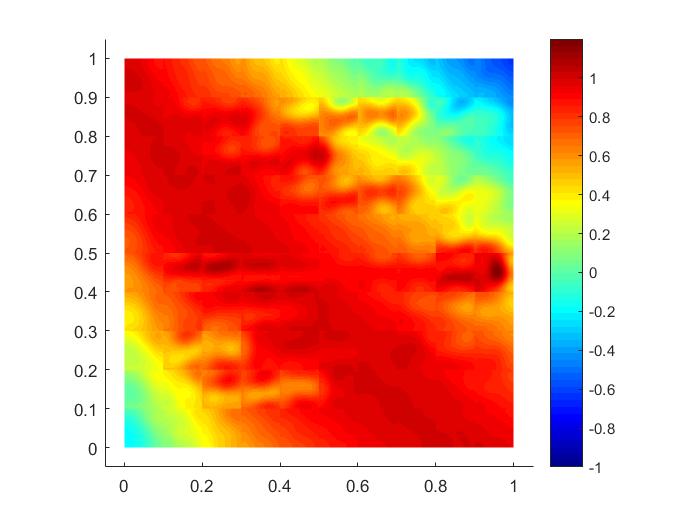}
		\title{$u_H \text{ with } L=10$}
	\end{minipage}\\%
	\begin{minipage}[t]{0.45\textwidth}
		\centering
		\includegraphics[width=3in]{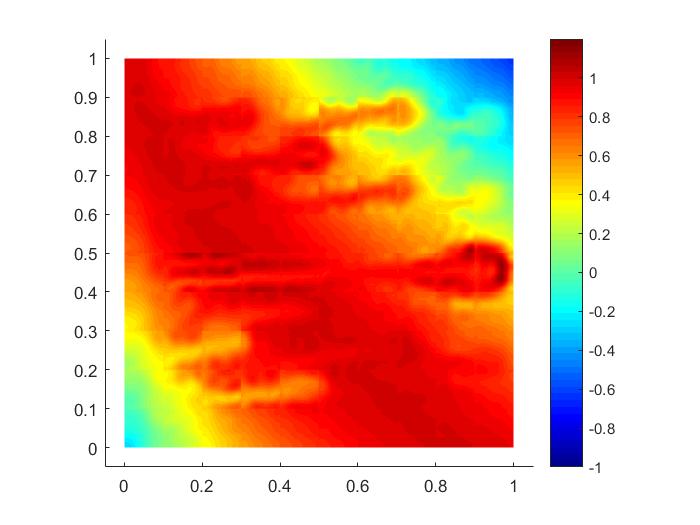}
		\title{$u_H \text{ with } L=20$}
	\end{minipage}%
	\begin{minipage}[t]{0.45\textwidth}
		\centering
		\includegraphics[width=3in]{UmsL20.jpg}
		\title{$u_H \text{ with } L=30$}
	\end{minipage}%
	\caption{$u_h$ and $u_H$ in Example 1.}
	\label{solution}
\end{figure}

In Figure \ref{inverse lambda}, we plot the values $1/\Lambda_*$, where $\Lambda_*=\min_{K_i}\lambda_{L_i+1}^{(i)}$, against the number of basis functions.
We clearly see the decay of the eigenvalues. We also observe that the 
decay is much faster for the first few eigenfunctions, which
implies that a few basis will give a substantial decay in error. 
In Table \ref{error table }, we show the errors using different 
numbers of offline basis functions $L_i$.
We see clearly the reduction of error when more basis functions are used,
and the reduction of error is more rapid when fewer basis functions are used.
We also observe that
the method gives reasonable error levels with small snapshot ratios. 
On the other hand, 
Figures \ref{solution} shows the fine and multiscale solutions at $t=0.08$. From these figures, we observe very good agreements between the
fine-scale and multiscale solutions.

In addition, we compare the performance of our method with the use of 
space-time polynomial basis. For space-time polynomial basis, we build local offline space $V_H^{i(n)}$ using $Q_s$ functions in $K_i\times (T_{n-1},T_n)$ (total $(s+1)^3$ functions), where $s=1,\,2,\,\cdots$ and $Q_s$ is the space of polynomials of degree $s$ in each direction.
We denote this solution using space-time polynomial basis 
by $u_\text{poly}$. Then, we compare these numerical results to  
GMsFEM method with $L_i=(s+1)^3$. In Table \ref{error compare polynomial}, 
we present the errors with the use of $s=1$ and $s=2$ 
for space-time polynomial basis and the
use of $L=8$ and $L=27$ multiscale basis. 
We note that the dimension of $V_H$ is the same for both cases. 
From this table, we see that the multiscale basis performs better than polynomial basis when the same number of basis is used. 
Figures \ref{compare polynomial} shows the corresponding solutions, and we observe that the GMsFEM provides better approximate solutions.

From the results in Tables \ref{error table } and \ref{error compare polynomial}, we observe our multiscale
approach provides an efficient representation of the solution. 
In particular, if one uses space-time piecewise linear approximation, the errors $e_1$ and $e_2$ are $6.79\%$ and $9.43\%$ respectively
and the dimension of the approximation space for each space-time cell is $8$.
On the other hand, the multiscale approach is able to obtain similar error levels by using $3$ multiscale basis functions per space-time cell. 
 Moreover,
 if one uses space-time piecewise quadratic approximation, the errors $e_1$ and $e_2$ are $4.12\%$ and $5.36\%$ respectively
and the dimension of the approximation space for each space-time cell is $27$.
On the other hand, the multiscale approach is able to obtain similar error levels by using $7$ multiscale basis functions per space-time cell. 

\begin{table}[ht]
	\centering 
	\begin{tabular}{|c|c|c|}
		\hline
		\hline
		& $e_1$ & $e_2$  \\
		\hline
		Multiscale basis with $L=8$  &  4.11\% & 5.69\% \\
		\hline
		Polynomial basis with $Q_s=Q_1$  &  6.79\% & 9.43\% \\
		\hline
		\hline
		Multiscale basis with $L=27$  &  1.95\% & 2.96\% \\
		\hline
		Polynomial basis with $Q_s=Q_2$  &  4.12\% & 5.36\%  \\
		\hline
	\end{tabular}
	\caption{Comparing the use of multiscale and polynomial basis functions for Example 1.}
	\label{error compare polynomial}
\end{table}

\begin{figure}[ht]
	\begin{minipage}[t]{0.45\textwidth}
		\centering
		\includegraphics[width=3in]{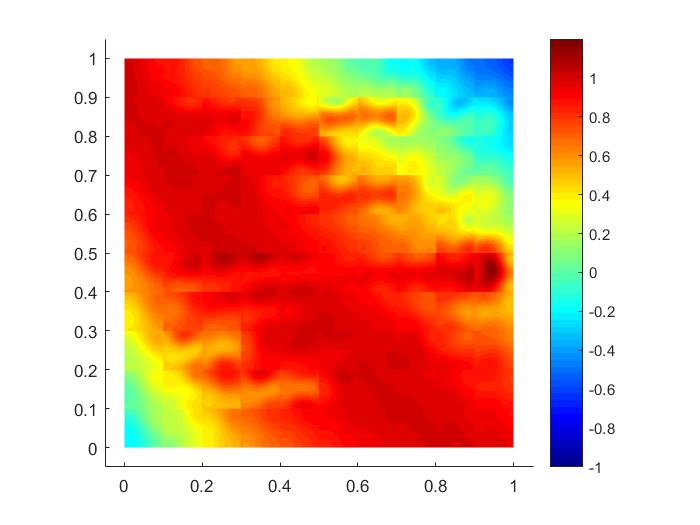}
		\title{$u_H \text{ with } L=8$}
	\end{minipage}%
	\begin{minipage}[t]{0.45\textwidth}
		\centering
		\includegraphics[width=3in]{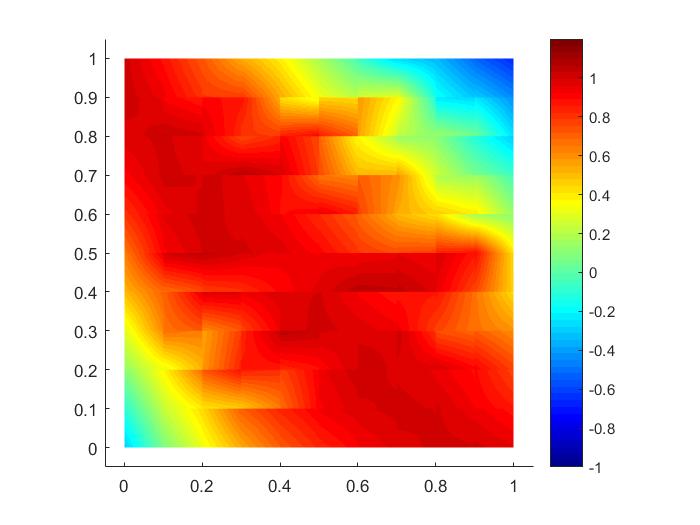}
		\title{$u_\text{poly} \text{ with } Q_s=Q_1$}
	\end{minipage}\\%
	\begin{minipage}[t]{0.45\textwidth}
		\centering
		\includegraphics[width=3in]{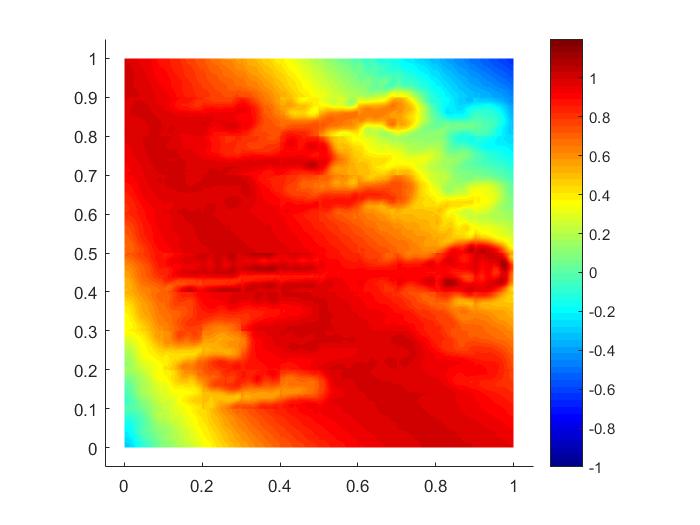}
		\title{$u_H \text{ with } L=27$}
	\end{minipage}%
	\begin{minipage}[t]{0.45\textwidth}
		\centering
		\includegraphics[width=3in]{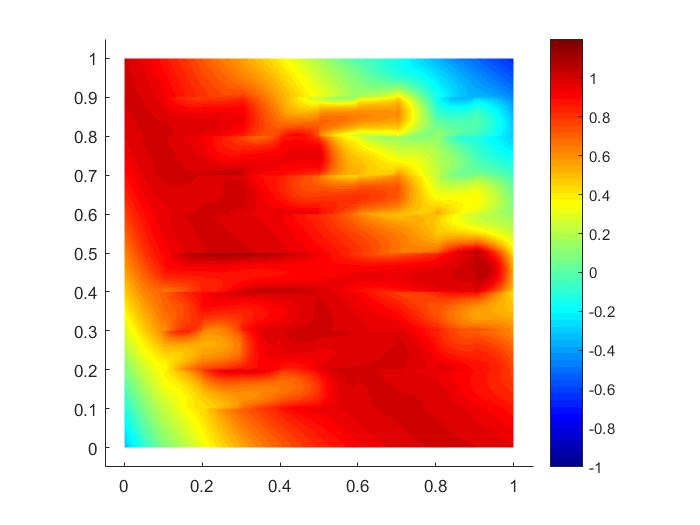}
		\title{$u_\text{poly} \text{ with } Q_s=Q_2$}
	\end{minipage}%
	\caption{Comparing $u_H$ with $u_\text{poly}$ in Example 1.}
	\label{compare polynomial}
\end{figure}

\subsection{Example 2}
In our second example, we also use CG in coarse cell case, take $u_0=1-xy$ and $g=1$. The velocity field $\mathbf{v}=(v_1,v_2)$ is the same as that in Example 1.
In Table \ref{ex2 error table }, we present the errors for using various choices of number of basis functions. 
We clearly see that, with a very small snapshot ratio, our method is able to obtain solutions with very good accuracy. 
Furthermore, we observe a faster decay of the error when smaller number of basis functions are used. This confirm the fast decay of eigenvalues
in the regime of smaller numbers of basis functions.
In Figures \ref{ex2 solution}, we present the fine and multiscale solutions
at the time $t=0.08$. We observe very good agreement of both solutions.


\begin{table}[ht]
	\centering 
	\begin{tabular}{|c|c|c|c|c|}
		\hline
		\hline
		$L$ & dim$(V_H^{(n)})$ & snapshot ratio & $e_1$ & $e_2$  \\
		\hline
		1 &100 & 0.45\%  &  44.82\% & 46.94\% \\
		\hline
		3 &300 & 1.34\%  &  3.96\% & 5.72\% \\
		\hline
		5 &500 & 2.24\%  &  3.39\% & 4.92\% \\
		\hline
		7 &700 & 3.13\%  &  2.28\% & 3.10\%  \\
		\hline
		10 &1000 & 4.47\%  &  1.97\% & 2.74\% \\
		\hline
		15 &1500 & 6.71\%  &  1.43\% & 2.21\% \\
		\hline
		20 &2000 & 8.94\% &  1.29\% & 1.86\% \\
		\hline
		25 &2500 & 11.18\% &  1.10\% & 1.65\% \\
		\hline
		30 &3000 & 13.41\%  &  1.02\% & 1.49\% \\
		\hline
	\end{tabular}
	\caption{Errors for Example 2 (dim$(V_h^{(n)})$=72600 and dim$(V_\text{snap}^{(n)})$=22365 for each time step $n$).}
	\label{ex2 error table }
\end{table}

\begin{figure}[ht]
	\begin{minipage}[t]{0.45\textwidth}
		\centering
		\includegraphics[width=3in]{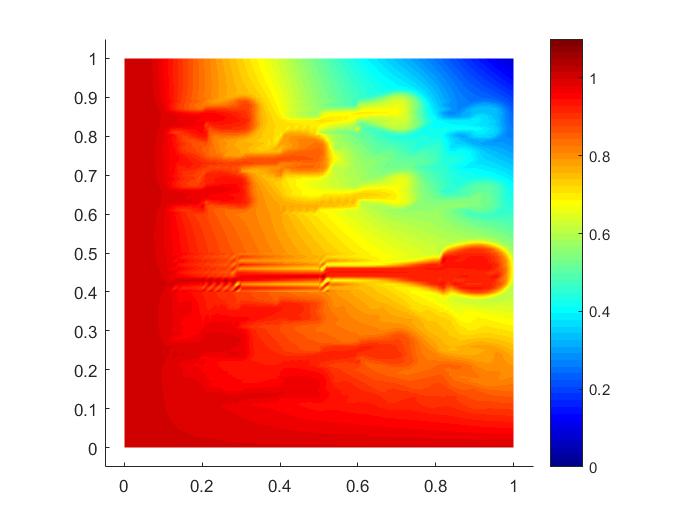}
		\title{$u_h$}
	\end{minipage}%
	\begin{minipage}[t]{0.45\textwidth}
		\centering
		\includegraphics[width=3in]{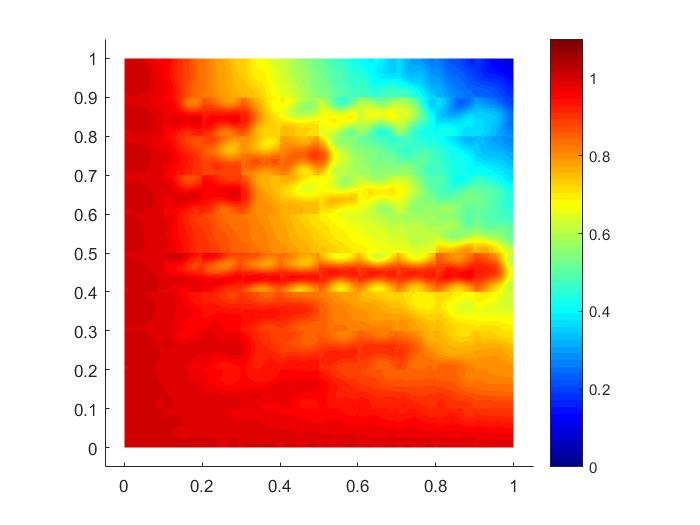}
		\title{$u_H \text{ with } L=10$}
	\end{minipage}\\%
	\begin{minipage}[t]{0.45\textwidth}
		\centering
		\includegraphics[width=3in]{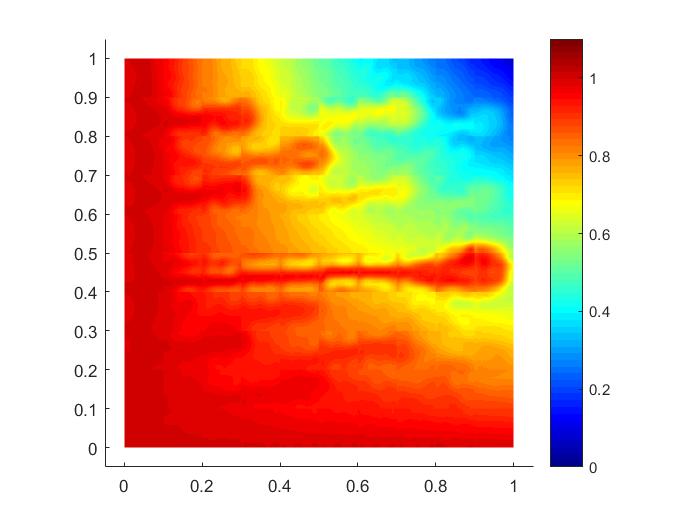}
		\title{$u_H \text{ with } L=20$}
	\end{minipage}%
	\begin{minipage}[t]{0.45\textwidth}
	\centering
	\includegraphics[width=3in]{ex2_UmsL20.jpg}
	\title{$u_H \text{ with } L=30$}
\end{minipage}%
	\caption{$u_h$ and $u_H$ in Example 2.}
	\label{ex2 solution}
\end{figure}

We also compare the performance of our method with the use of space-time polynomial basis functions,
and the results are presented in Table \ref{ex2 error compare polynomial}
and Figures \ref{ex2 compare polynomial}.
We observe similar conclusions as in the first example.
In particular, we see that the multiscale basis functions give more accurate solutions
compared with the polynomial basis functions when the same numbers of basis functions are used. 
We also see from Tables \ref{ex2 error table } and \ref{ex2 error compare polynomial}
that multiscale basis functions give faster error decay. 
For the $e_1$ error of about $3.8\%$, our multiscale method needs only $3$ basis functions
while the use of polynomial needs $8$ basis functions. Besides, for 
the $e_1$ error of about $2.3\%$, our multiscale method needs only $7$ basis functions
while the use of polynomial needs $27$ basis functions.
So, we see the rapid decay of error by using multiscale basis functions.

\begin{table}[ht]
	\centering 
	\begin{tabular}{|c|c|c|}
		\hline
		\hline
		& $e_1$ & $e_2$  \\
		\hline
		Multiscale basis with $L=8$  &  2.26\% & 3.11\% \\
		\hline
		Polynomial basis with $Q_s=Q_1$  &  3.85\% & 5.62\% \\
		\hline
		\hline
		Multiscale basis with $L=27$  &  1.07\% & 1.59\% \\
		\hline
		Polynomial basis with $Q_s=Q_2$  &  2.46\% & 3.23\%  \\
		\hline
	\end{tabular}
	\caption{Comparing the use of multiscale and polynomial basis functions for Example 2.}
	\label{ex2 error compare polynomial}
\end{table}

\begin{figure}[ht]
	\begin{minipage}[t]{0.45\textwidth}
		\centering
		\includegraphics[width=3in]{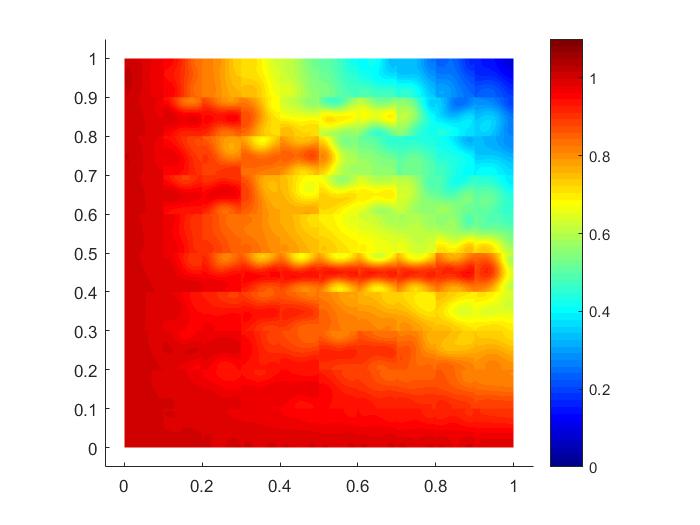}
		\title{$u_H \text{ with } L=8$}
	\end{minipage}%
	\begin{minipage}[t]{0.45\textwidth}
		\centering
		\includegraphics[width=3in]{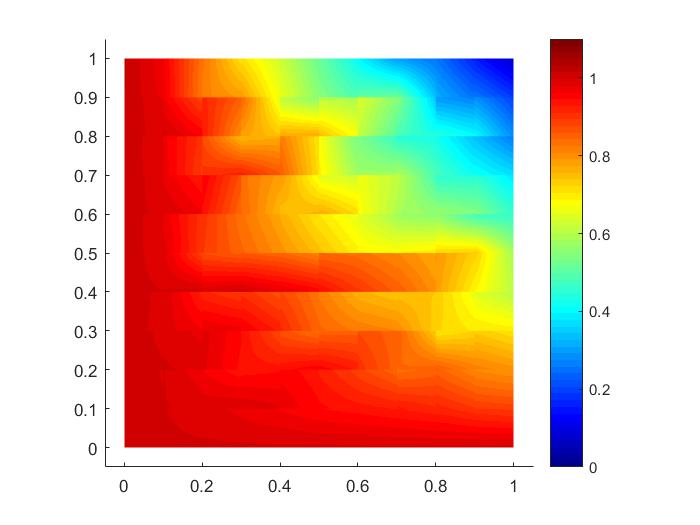}
		\title{$u_\text{poly} \text{ with } Q_s=Q_1$}
	\end{minipage}\\%
	\begin{minipage}[t]{0.45\textwidth}
		\centering
		\includegraphics[width=3in]{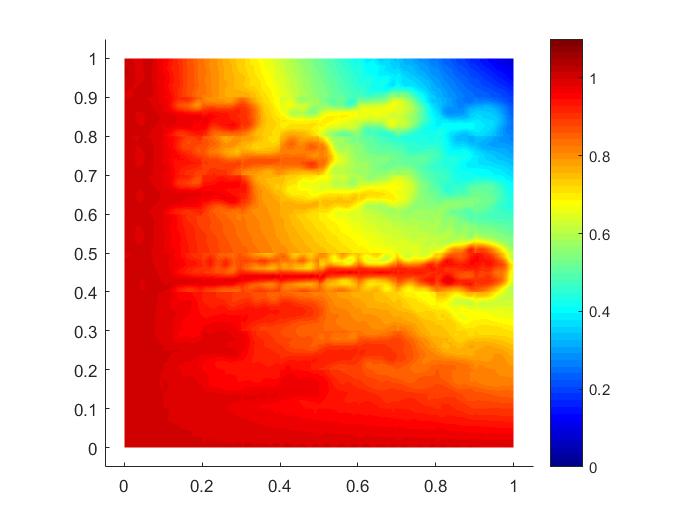}
		\title{$u_H \text{ with } L=27$}
	\end{minipage}%
	\begin{minipage}[t]{0.45\textwidth}
		\centering
		\includegraphics[width=3in]{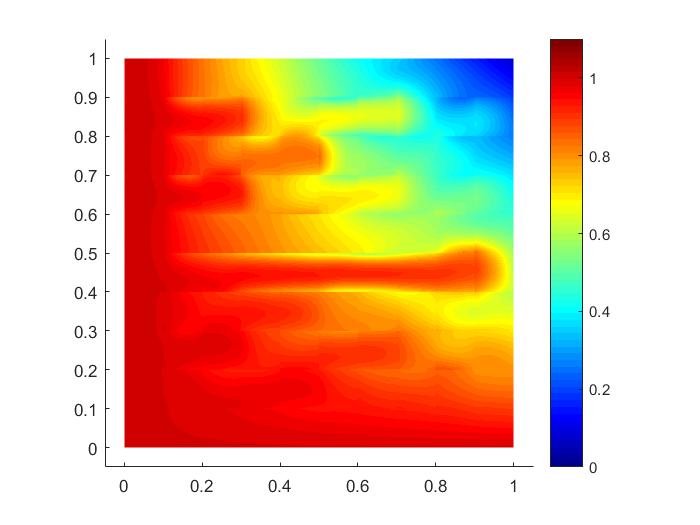}
		\title{$u_\text{poly} \text{ with } Q_s=Q_2$}
	\end{minipage}%
	\caption{Comparing $u_H$ with $u_\text{poly}$ in Example 2.}
	\label{ex2 compare polynomial}
\end{figure}

\section{\textbf{Conclusion}}

In this paper, we consider the construction of the space-time GMsFEM to solve time dependent 
transport equation with heterogeneous velocity field. 
To our best knowledge, this is a first attempt to generate 
space-time multiscale basis functions for convection problems,
that are known to be challenging because of strong distant effects.
Our main objective is to develop systematic multiscale model reduction 
techniques in space-time cells by constructing local (in space-time) 
multiscale basis functions. The proposed concepts can be
used for other applications, where one needs space-time multiscale basis functions.
{\color{black}{
Our approach focuses on
(1) constructing space-time snapshot
vectors, 
(2) performing appropriate t
local spectral decomposition in the snapshot space. 
For  snapshot vectors, we
solve local problems in local space-time domains. 
A complete snapshot space includes solutions with all
possible boundary and initial conditions. 
Local spectral decomposition is derived from the analysis.
We present a convergence analysis of the proposed method
and show that one can obtain a stable and robust multiscale discretization.
Several
numerical examples are presented. 
We consider examples where the velocity
fields are highly heterogeneous in the space. With only spatial multiscale
basis functions are used, we will need  a large dimensional space. 
The space-time multiscale space allows reducing the degrees of freedom.
Our numerical results show that one can
obtain accurate solutions.
}}
Though the presented results are promising, there is a room for further 
improvements. In particular, we will seek more accurate multiscale basis
functions and develop online approaches \cite{chung2015residual}. The main
idea of online approaches is to add multiscale basis functions
using the residual information. With appropriate offline spaces,
one can achieve a fast convergence with online basis functions. This
will be studied in our future work.


\section{Acknowledgements}

The authors are grateful to Wing Tat Leung for many helpful suggestions.

\bibliographystyle{plain}
\bibliography{references,references1,references2,references_outline}

\begin{thebibliography}{10}

\bibitem{aej07}
J.E. Aarnes, Y.~Efendiev, and L.~Jiang.
\newblock Analysis of multiscale finite element methods using global
  information for two-phase flow simulations.
\newblock {\em SIAM J. Multiscale Modeling and Simulation}, 7:2177--2193, 2008.

\bibitem{Arbogast_two_scale_04}
T.~Arbogast.
\newblock Analysis of a two-scale, locally conservative subgrid upscaling for
  elliptic problems.
\newblock {\em SIAM J. Numer. Anal.}, 42(2):576--598 (electronic), 2004.

\bibitem{Arbogast_Boyd_06}
T.~Arbogast and K.J. Boyd.
\newblock Subgrid upscaling and mixed multiscale finite elements.
\newblock {\em SIAM J. Numer. Anal.}, 44(3):1150--1171 (electronic), 2006.

\bibitem{ArPeWY07}
T.~Arbogast, G.~Pencheva, M.F. Wheeler, and I.~Yotov.
\newblock A multiscale mortar mixed finite element method.
\newblock {\em Multiscale Model. Simul.}, 6(1):319--346 (electronic), 2007.

\bibitem{Arbogast_PWY_07}
T.~Arbogast, G.~Pencheva, M.F. Wheeler, and I.~Yotov.
\newblock A multiscale mortar mixed finite element method.
\newblock {\em Multiscale Model. Simul.}, 6(1):319--346, 2007.

\bibitem{brandt05}
A.~Brandt.
\newblock Multiscale solvers and systematic upscaling in computational physics.
\newblock {\em Computer Physics Communication}, 169:438--441, 2005.

\bibitem{chen2016least}
Fuchen Chen, Eric Chung, and Lijian Jiang.
\newblock Least-squares mixed generalized multiscale finite element method.
\newblock {\em Computer Methods in Applied Mechanics and Engineering},
  311:764--787, 2016.

\bibitem{cdgw03}
Y.~Chen, L.~Durlofsky, M.~Gerritsen, and X.~Wen.
\newblock A coupled local-global upscaling approach for simulating flow in
  highly heterogeneous formations.
\newblock {\em Advances in Water Resources}, 26:1041--1060, 2003.

\bibitem{Chu_Hou_MathComp_10}
C.-C. Chu, I.~G. Graham, and T.-Y. Hou.
\newblock A new multiscale finite element method for high-contrast elliptic
  interface problems.
\newblock {\em Math. Comp.}, 79(272):1915--1955, 2010.

\bibitem{MixedGMsFEM}
E.~Chung, Y.~Efendiev, and C.~Lee.
\newblock Mixed generalized multiscale finite element methods and applications.
\newblock {\em SIAM Multicale Model. Simul.}, 13:338--366, 2014.

\bibitem{AdaptiveGMsFEM}
E.~T. Chung, Y.~Efendiev, and G.~Li.
\newblock An adaptive {GM}s{FEM} for high contrast flow problems.
\newblock {\em J. Comput. Phys.}, 273:54--76, 2014.

\bibitem{ceh2016adaptive}
Eric Chung, Yalchin Efendiev, and Thomas~Y Hou.
\newblock Adaptive multiscale model reduction with generalized multiscale
  finite element methods.
\newblock {\em Journal of Computational Physics}, 320:69--95, 2016.

\bibitem{chung2015residual}
Eric~T Chung, Yalchin Efendiev, and Wing~Tat Leung.
\newblock Residual-driven online generalized multiscale finite element methods.
\newblock {\em Journal of Computational Physics}, 302:176--190, 2015.

\bibitem{chung2016spacetime}
Eric~T Chung, Yalchin Efendiev, Wing~Tat Leung, and Shuai Ye.
\newblock Generalized multiscale finite element methods for space-time
  heterogeneous parabolic equations.
\newblock {\em arXiv preprint arXiv:1605.07634}, 2016.

\bibitem{chung2017conservative}
Eric~T Chung, Maria Vasilyeva, and Yating Wang.
\newblock A conservative local multiscale model reduction technique for stokes
  flows in heterogeneous perforated domains.
\newblock {\em Journal of Computational and Applied Mathematics}, 2017.

\bibitem{EG09}
Y.~Efendiev and J.~Galvis.
\newblock Domain decomposition preconditioner for multiscale high-contrast
  problems.
\newblock In {\em Proceedings of DD19}, 2009.

\bibitem{Efendiev_Galvis_10}
Y.~Efendiev and J.~Galvis.
\newblock Eigenfunctions and multiscale methods for {D}arcy problems.
\newblock Technical report, ISC, Texas A\& M University, 2010.

\bibitem{egh12}
Y.~Efendiev, J.~Galvis, and T.~Hou.
\newblock Generalized multiscale finite element methods.
\newblock {\em Journal of Computational Physics}, 251:116--135, 2013.

\bibitem{Efendiev_GKiL_12}
Y.~Efendiev, J.~Galvis, S.~Ki~Kang, and R.D. Lazarov.
\newblock Robust multiscale iterative solvers for nonlinear flows in highly
  heterogeneous media.
\newblock {\em Numer. Math. Theory Methods Appl.}, 5(3):359--383, 2012.

\bibitem{Ewing_ILRW_SISC_09}
R.~Ewing, O.~Iliev, R.~Lazarov, I.~Rybak, and J.~Willems.
\newblock A simplified method for upscaling composite materials with high
  contrast of the conductivity.
\newblock {\em SIAM J. Sci. Comput.}, 31(4):2568--2586, 2009.

\bibitem{hou1992homogenization}
Thomas~Y Hou and Xue Xin.
\newblock Homogenization of linear transport equations with oscillatory vector
  fields.
\newblock {\em SIAM Journal on Applied Mathematics}, 52(1):34--45, 1992.

\bibitem{iliev2013numerical}
O~Iliev, Z~Lakdawala, and V~Starikovicius.
\newblock On a numerical subgrid upscaling algorithm for stokes--brinkman
  equations.
\newblock {\em Computers \& Mathematics with Applications}, 65(3):435--448,
  2013.

\bibitem{Iliev_fast_fibrous_10}
O.~Iliev, R.~Lazarov, and J.~Willems.
\newblock Fast numerical upscaling of heat equation for fibrous materials.
\newblock {\em Comput. Vis. Sci.}, 13(6):275--285, 2010.

\bibitem{Iliev_MMS_11}
O.~Iliev, R.~Lazarov, and J.~Willems.
\newblock Variational multiscale finite element method for flows in highly
  porous media.
\newblock {\em Multiscale Model. Simul.}, 9(4):1350--1372, 2011.

\bibitem{Iliev_RW_09}
O.P. Iliev, I.~Rybak, and J.~Willems.
\newblock On upscaling of heat conductivity for a class of industrial problems.
\newblock {\em J. Theoretical and Applied Mechanics}.
\newblock submitted.

\bibitem{riviere2008discontinuous}
B{\'e}atrice Rivi{\`e}re.
\newblock {\em Discontinuous Galerkin methods for solving elliptic and
  parabolic equations: theory and implementation}.
\newblock Society for Industrial and Applied Mathematics, 2008.

\bibitem{Vass_Upscal_11}
P.S. Vassilevski.
\newblock Coarse spaces by algebraic multigrid: multigrid convergence and
  upscaling error estimates.
\newblock {\em Adv. Adapt. Data Anal.}, 3(1-2):229--249, 2011.

\bibitem{weinan1992homogenization}
E~Weinan.
\newblock Homogenization of linear and nonlinear transport equations.
\newblock {\em Communications on Pure and Applied Mathematics}, 45(3):301--326,
  1992.

\bibitem{weh02}
X.H. Wu, Y.~Efendiev, and T.Y. Hou.
\newblock Analysis of upscaling absolute permeability.
\newblock {\em Discrete and Continuous Dynamical Systems, Series B.},
  2:158--204, 2002.

\end{thebibliography}

\end{document}